\documentclass[a4paper,twoside,12pt]{article}  

\usepackage{a4wide}
\usepackage{graphicx}
\usepackage{amsmath,amssymb,amsthm, amsgen}
\usepackage{listings}
\usepackage{booktabs}
\usepackage{color}
\usepackage[usenames,dvipsnames]{xcolor}
\usepackage{rotating} 
\usepackage{enumerate}
\usepackage{hhline}
\usepackage{multirow}
\usepackage{enumitem}   
\usepackage[bookmarks]{}
\usepackage{dsfont}
\usepackage{tikz}
\usepackage{esint}
\usepackage{booktabs}
\usepackage{stmaryrd}
\usepackage{url}
\usepackage{fancyhdr}

\newtheorem{thm}{Theorem}[section]
\newtheorem{prop}[thm]{Proposition}
\newtheorem{ass}[thm]{Assumption}

\newtheorem{cor}[thm]{Corollary}

\theoremstyle{definition}
\newtheorem{rem}[thm]{Remark}

\newcommand{\bmu}{\boldsymbol \mu}

\newcommand{\C}{\mathbb C}
\newcommand{\Curl}{\operatorname{Curl}}

\newcommand{\Div}{\operatorname{Div}}
\newcommand{\Id}{\operatorname{Id}}
\newcommand{\dist}{\operatorname{dist}}

\newcommand{\N}{\mathbb N}
\newcommand{\R}{\mathbb R}
\newcommand{\sym}{\operatorname{sym}}

\newcommand{\indicatornoacc}[1]{ \mathds{1}_{ #1 } }
\newcommand{\tr}{\operatorname{tr}}
\newcommand{\supp}{\operatorname{supp}}
\newcommand{\weakto}{\rightharpoonup}
\newcommand{\xto}[1]{\xrightarrow{ #1 }}
\newcommand{\xweakto}[1]{ \stackrel{ #1 }{\rightharpoonup} }
\newcommand{\Z}{\mathbb Z}
\newcommand{\Np}{\mathbb{N}_+}

\newcommand{\e}{\varepsilon}
\newcommand{\bb}{\mathbf b}

\newcommand{\bx}{\mathbf x}

\newcommand{\cV}{\mathcal V}

\newcommand{\Vreg}{V_{\operatorname{reg}}}
\newcommand{\Vregt}{\widetilde V_{\operatorname{reg}}}

\newcommand{\lrhaa}[1]{\left( #1 \right)}
\newcommand{\Bighaa}[1]{\Big( #1 \Big)}
\newcommand{\bighaa}[1]{\big( #1 \big)}

\def\weakto{\rightharpoonup}

\DeclareFontFamily{U}{mathx}{\hyphenchar\font45}
\DeclareFontShape{U}{mathx}{m}{n}{
      <5> <6> <7> <8> <9> <10>
      <10.95> <12> <14.4> <17.28> <20.74> <24.88>
      mathx10
      }{}
\DeclareSymbolFont{mathx}{U}{mathx}{m}{n}
\DeclareFontSubstitution{U}{mathx}{m}{n}
\DeclareMathAccent{\widecheck}{0}{mathx}{"71}

\begin{document}

\title{The continuum limit of interacting dislocations \\ on multiple slip systems} 

\author{Patrick van Meurs} 



%

\maketitle

\begin{abstract} 
In this paper we derive the continuum limit of a multiple-species, interacting particle system by proving a $\Gamma$-convergence result on the interaction energy as the number of particles tends to infinity. As the leading application, we consider $n$ edge dislocations in multiple slip systems. Since the interaction potential of dislocations has a logarithmic singularity at zero with a sign that depends on the orientation of the slip systems, the interaction energy is unbounded from below. To make the minimization problem of this energy meaningful, we follow the common approach to regularise the interaction potential over a length-scale $\delta > 0$. The novelty of our result is that we leave the \emph{type} of regularisation general, and that we consider the joint limit $n \to \infty$ and $\delta \to 0$. Our result shows that the limit behaviour of the interaction energy is not affected by the type of the regularisation used, but that it may depend on how fast the \emph{size} (i.e., $\delta$) decays as $n \to \infty$.
\end{abstract}



\section{Introduction}
\label{s:intro}

In this paper we study multiple-species particle systems with singular interactions. We focus on the leading application to edge \emph{dislocations} in crystals with multiple slip directions. Dislocations are defects in the crystallographic lattice of metals, and the concerted motion of a large number of such defects gives rise to plastic deformation at the macroscopic scale. Yet, a satisfactory connection between the microscopic description of a large number of dislocations and macroscopic models for plastic deformation remains elusive. This paper takes a next step in clarifying this connection.

To take this next step, we consider a simplified microscopic model for the interaction energy of a collection of dislocations. In this model, dislocations are modelled as point defects in a continuum medium in two dimensions. Our objective is to derive a continuum energy as the number of dislocations tends to infinity, where dislocations are described in terms of a continuum density. 

Our objective fits to the vast literature on the derivation of continuum, density-based models from interacting particle systems, both in deterministic and stochastic settings. In most studies the particles are identical (single-species), and their interactions can be attractive, repulsive (or a combination of the two), smooth or singular, radial or anisotropic. The derivation of such continuum models is far less studied for particle systems involving multiple-species, especially for singular interactions.  This has motivated us to go beyond the application to dislocations by starting from a more general multiple-species, interacting particle energy, and to derive the continuum limit thereof. Besides the application to edge dislocations, we also show that our setting applies to screw dislocations, vortices and particle systems interacting by Riesz-like potentials.

\subsection{Formal setup for the case of edge dislocations}

We consider Volterra's model for $n \in \N$ straight and parallel edge dislocations. In this model, each dislocation (labelled by $i = 1,\ldots,n$) is characterised by a couple $(x_i, b_i)$, where $x_i \in \R^2$ is its position and $b_i$ is its Burgers vector in the unit circle $\mathbb S$. Figure \ref{fig:sett} illustrates a possible configuration. While dislocations are mobile, their Burgers vectors are fixed; hence, we regard $\bx=(x_1,\dots,x_n) \in (\R^2)^n$ as the unknown and $\bb = (b_1,\dots,b_n) \in \mathbb S^n$ as a list of given parameters.

\begin{figure}[h]
\centering
\begin{tikzpicture}[scale=.8]
    \def \r {0.15}
    \def \dlcScale {0.25}
      
    \begin{scope}[shift={(1.5,0)},scale=\dlcScale, rotate=0]
        \fill (-1, -\r) rectangle (1, \r); 
        \fill (-\r, 0) rectangle (\r, 2); 
        \fill (-1, 0) circle (\r);
        \fill (1, 0) circle (\r);
        \fill (0, 2) circle (\r);  
        \draw[->] (0,0) -- (3,0) node[right]{$b_2$};
        \draw (0,0) node[below]{$x_2$};
    \end{scope}
    
    \begin{scope}[shift={(0,2)},scale=\dlcScale, rotate=60]
        \fill (-1, -\r) rectangle (1, \r); 
        \fill (-\r, 0) rectangle (\r, 2); 
        \fill (-1, 0) circle (\r);
        \fill (1, 0) circle (\r);
        \fill (0, 2) circle (\r);  
        \draw[->] (0,0) -- (3,0) node[right]{$b_3$};
        \draw (0,0) node[right]{$x_3$};
    \end{scope}
    
    \begin{scope}[shift={(-1.15,0)},scale=\dlcScale, rotate=240]
        \fill (-1, -\r) rectangle (1, \r); 
        \fill (-\r, 0) rectangle (\r, 2); 
        \fill (-1, 0) circle (\r);
        \fill (1, 0) circle (\r);
        \fill (0, 2) circle (\r);
        \draw (0,0) node[left]{$x_1$};   
        \draw[->] (0,0) -- (3,0) node[left]{$b_1$};
    \end{scope}
    
    \begin{scope}[shift={(-2,2)},scale=\dlcScale, rotate=60]
        \fill (-1, -\r) rectangle (1, \r); 
        \fill (-\r, 0) rectangle (\r, 2); 
        \fill (-1, 0) circle (\r);
        \fill (1, 0) circle (\r);
        \fill (0, 2) circle (\r);  
        \draw (0,0) node[right]{$x_5$};
        \draw[->] (0,0) -- (3,0) node[right]{$b_5$};
    \end{scope}
    
    \begin{scope}[shift={(-4,1)},scale=\dlcScale, rotate=0]
        \fill (-1, -\r) rectangle (1, \r); 
        \fill (-\r, 0) rectangle (\r, 2); 
        \fill (-1, 0) circle (\r);
        \fill (1, 0) circle (\r);
        \fill (0, 2) circle (\r);  
        \draw (0,0) node[below]{$x_4$};
        \draw[->] (0,0) -- (3,0) node[right]{$b_4$};
    \end{scope}
    
    \begin{scope}[shift={(-9,1)}, scale=0.7]
        \draw[dotted] (0,0) circle (2);
        \fill (0, 0) circle (\r); 
        \draw (-2,0) node[right]{$\mathbb S$};
        \draw[->] (0,0) -- (2,0) node[right]{$\xi_1$};
        \draw[->,rotate=60] (0,0) -- (2,0) node[above]{$\xi_2$};
	   \draw[->, rotate=240] (0,0) -- (2,0) node[below]{$\xi_3$};
    \end{scope}
\end{tikzpicture} 
\caption{Example of a configuration of $n = 5$ many edge dislocations denoted by `$\perp$'. The set $\{ \xi_1, \xi_2, \xi_3 \} \subset \mathbb S$ is a given set of admissible Burgers vectors.}
\label{fig:sett}
\end{figure}
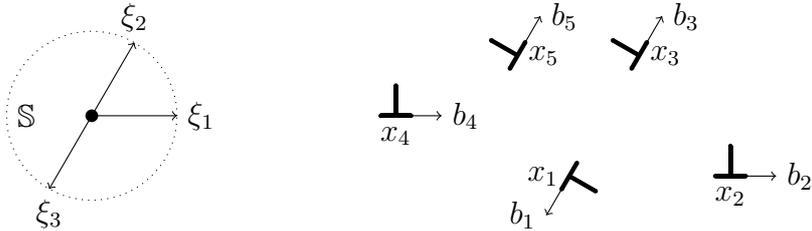

The interaction energy for Volterra dislocations is given by 
\begin{align} \label{Entil:unreg}
  \tilde E_n (\bx; \bb) 
  := \frac1{2n^2} \sum_{i = 1}^n \sum_{\substack{ j = 1 \\ j \neq i } }^n V (x_i - x_j;b_i,b_j),
\end{align}
where $V$ is the pairwise interaction potential between any two dislocations $(x_i,b_i)$ and $(x_j,b_j)$. It is given by (see, e.g., \cite[(5--16)]{HirthLothe82})
\begin{equation} \label{V:bibj}
    V(x; b_i,b_j) 
    := - b_i\cdot b_j \log | x | - \left(b_i^\perp\cdot \frac{x}{|x|}\right)\left(b_j^\perp\cdot \frac{x}{|x|}\right)
    \quad \text{for all } x \in \R^2 \setminus \{0\},
  \end{equation}
where $b^\perp$ denotes the clockwise rotation of $b$ by $\pi/2$. This expression is related to the Green's function of the elasticity operator in an isotropic medium. Consequently, it has a logarithmic singularity at the origin. The sign of the singularity depends on $b_i \cdot b_j$. Hence, aside from special cases, $\bx \mapsto \tilde E_n (\bx; \bb)$ is unbounded from below, even when all $x_i$ are confined to a compact set. Thus, to obtain a meaningful interaction energy for describing the dislocation positions, we need to alter the definition of $\tilde E_n$.

\subsection{The regularised interaction energy}
\label{s:intro:Vreg}

The unboundedness of $\tilde E_n$ from below is a consequence of the singularity of $V$, which is an artefact from Volterra's dislocation model. In fact, more detailed, atomistic descriptions for the dislocation interaction potential are not singular at $0$; see, e.g., \cite[Sec.~5.3]{ArizaOrtiz05}. In order to remove the singularity at zero without adding the complexity of atomistic effects, the common approach is to \emph{regularise} $V$ over an atomic length scale $\delta > 0$. This has been done in various ways: 
\begin{itemize}
  \item[-] via the phase-field model developed by Peierls and Nabarro \cite{GarroniMueller06,KoslowskiCuitinoOrtiz02, MonneauPatrizi12, Nabarro47,Peierls40}; 
  \item[-] by perforating the elastic medium around each dislocation \cite{CermelliLeoni06, GarroniLeoniPonsiglione10, MoraPeletierScardia17}; 
  \item[-] by smearing out of the dislocation core by a convolution kernel \cite{AlvarezCarliniHochLBouarMonneau05, CaiArsenlisWeinbergerBulatov06, ContiGarroniOrtiz15}, or, equivalently, by convoluting $V$ by the related kernel; 
  \item[-] via a cut-off radius within which dislocations do not interact \cite{HirthLothe82}. 
\end{itemize}
The choice for the regularisation $V_\delta$ of $V$ depends on factors such as accuracy with respect to atomistic models, computational convenience, the possibility to describe dynamics, well-posedness or the possibility to establish a continuum limit. However, there is no consensus on which regularisation works best, and each available regularisation has a set of drawbacks (see, e.g., the overview in \cite{CaiArsenlisWeinbergerBulatov06}). For instance, a regularisation which leads to a computationally cheap model may fail in terms of accuracy.

Because of the different available choices for $V_\delta$, we set out to identify a large class of such choices for which the continuum limit passage as $n \to \infty$ yields a meaningful limit energy, and for which the limit is independent of the choice for $V_\delta$. This class of regularisations is detailed in Assumption \ref{a:Vd}; in Section \ref{s:intro:results:VVd} we highlight the key structural assumption. Since we interpret the limit passage $n \to \infty$ as zooming out from the material, we assume that the atomistic length-scale $\delta = \delta_n$ converges to $0$ as $n \to \infty$. 

Given $\delta_n \to 0$ as $n \to \infty$ and a regularisation $V_{\delta_n}$, the corresponding regularised interaction energy is
\begin{align} \label{En:sum}
E_n (\bx; \bb) 
  := \frac1{n^2} \sum_{i=1}^n \sum_{j=1}^n V_{\delta_n} (x_i - x_j;b_i,b_j).
\end{align}
This is the energy for a collection of edge dislocations which we cover in our main result on passing to the limit $n\to \infty$. In contrast to $\tilde E_n$, we have multiplied in \eqref{En:sum} by $2$ for notational convenience. More importantly, we have included the diagonal in the double sum. This diagonal corresponds to a constant contribution to the energy given by
\begin{equation} \label{gamman:bi}
  \gamma_n := \frac1{n^2} \sum_{i=1}^n V_{\delta_n} (0;b_i,b_i).
\end{equation}
In terms of modelling, $\gamma_n$ may be interpreted as a self-energy for dislocations. Mathematically, the condition $\gamma_n \to 0$ as $n \to \infty$ turns out to be crucial for our main convergence result to hold. This condition turns for most regularisations $V_{\delta_n}$ into a lower bound on $\delta_n$. Indeed, in view of \eqref{V:bibj}, a typical regularisation (for instance, a regularisation by convolution with a mollifier supported in the ball $B(0, \delta_n)$) satisfies $V_{\delta_n}(0; b, b) \sim \log \frac1{\delta_n}$. Then, the condition $\gamma_n \to 0$ is equivalent to
\begin{equation} \label{deln:regime}
  \log \frac1{\delta_n} \ll n.
\end{equation}

\subsection{Main result: the limit $n \to \infty$}
\label{s:intro:results}

We start by recalling previous convergence results as $n \to \infty$ for systems of $n$ dislocations: 
\begin{itemize}
  \item[-] in one spatial dimension (in particular, single-slip), both in the single-sign case (see, e.g., \cite{FocardiGarroni07, 
ForcadelImbertMonneau08, 
ForcadelImbertMonneau09, 
HajjIbrahimMonneau09, 
HallChapmanOckendon10, 
Hall11, 
GeersPeerlingsPeletierScardia13, 
VanMeursMunteanPeletier14, 
GarroniVanMeursPeletierScardia16, 
HallHudsonVanMeurs18}), 
and in the multiple-sign case (see \cite{ChapmanXiangZhu15, 
vanMeurs18});

\item[-] in two dimensions in the single-slip and single-sign case \cite{MoraPeletierScardia17}, and in the multiple-slip case \cite{GarroniLeoniPonsiglione10,
Ginster19,
LauteriLuckhaus16ArXiv,
MullerScardiaZeppieri14};  

\item[-] for the corresponding gradient flows \cite{AlicandroDeLucaGarroniPonsiglione14, AlicandroDeLucaGarroniPonsiglione16,ElHajjIbrahimMonneau09, 
ForcadelImbertMonneau12, 
VanMeursMorandotti19,
VanMeursMuntean14,
MonneauPatrizi12}
 and rate independent flows \cite{MoraPeletierScardia17}.
\end{itemize}
In particular, the result in \cite{GarroniLeoniPonsiglione10} is closely related to ours; it also provides a convergence result for a large collection of edge dislocations. However, \cite{GarroniLeoniPonsiglione10} considers a different energy $E_n$ which depends on both $\bx$ and a strain field $\beta : \R^2 \to \R^{2\times 2}$. We give a proper comparison with this energy and the related convergence results in Section \ref{s:intro:disc}. Instead of considering strain fields, we establish in this paper the convergence of $E_n$ from the viewpoint of particle systems, and show that our result applies to a large class of interaction potentials $V$ and to several regularisations $V_{\delta_n}$ thereof. In particular, this class covers integrable singularities of various strengths. Therefore, we believe our result to be interesting for general multiple-species, interacting particle system such as those in recent studies
\cite{BerendsenBurgerPietschmann17, 
DiFrancescoFagioli13, 
DiFrancescoFagioli16,
EversFetecauKolokolnikov17,
GarroniVanMeursPeletierScardia19DOI,
vanMeurs18,
VanMeursMorandotti19,
Zinsl16}. 

\subsubsection{The key assumption on $V$ and $V_{\delta_n}$} 
\label{s:intro:results:VVd}

We return our focus to the case of edge dislocations with $V$ as in \eqref{V:bibj}. One of the main challenges for passing to the limit $n \to \infty$ in the energy $E_n$ is to prove a lower bound on $E_n$ which is uniform in $n$. The origin of this difficulty is the fact that $V( x ; b_i, b_j)$ is unbounded from below whenever $b_i \cdot b_j \neq 0$. More precisely, if $b_i \cdot b_j > 0$, then $V( x ; b_i, b_j) \to -\infty$ as  $|x| \to \infty$; if $b_i \cdot b_j < 0$, then $V( x ; b_i, b_j) \to -\infty$ as $|x| \to 0$. Due to the minimal requirement that $V_{\delta_n} \to V$ as $n \to \infty$, it is non-trivial to find a uniform lower bound on $E_n$. The unboundedness of the tails of $V$ can be dealt with in several well-known ways; we choose to deal with this later on in the choice of topology in our $\Gamma$-convergence result as $n \to \infty$.

The singularity of $V$ at $0$ requires special care when proving a lower bound on $E_n$ for a suitable $V_{\delta_n}$. The key feature of $V$ which we use to prove such a lower bound is that it can be decomposed as 
\begin{equation} \label{V:W-form-intro}
  V(x;b_i,b_j) = \sum_{k=1}^{\mathsf K} \left( \overline W_k^{b_i} \ast W_k^{b_j} \right) (x) + \Vreg(x ;b_i,b_j)
\end{equation}
for some $\mathsf K \geq 1$, $W_k^b \in L^1(\R^2)$ and $\Vreg \in C(\R^2)$, where $\overline W_k^b (x) := W_k^b (-x)$. In this decomposition, the convolution terms carry the singularity and the function $\Vreg \in C(\R^2)$ is a regular remainder which describes the possibly non-integrable tails of $V$. Inspired by \cite{CermelliLeoni06}, we construct such a decomposition in Proposition \ref{p:V:CKK} with $\mathsf K = 4$. 

For the class of regularisations $V_{\delta_n}$ (specified later in Assumption \ref{a:Vd}) we require the same kind of decomposition as in \eqref{V:W-form-intro}. More precisely, we assume that there exist regularisations $W_{{\delta_n},k}^{b_i}$ and $\Vreg^{\delta_n}$ of $W_{k}^{b_i}$ and $\Vreg$ such that 
\begin{equation} \label{V:W-form-d-intro}
  V_{\delta_n}(x ;b_i,b_j) 
  = \sum_{k=1}^{\mathsf K} \left( \overline W_{{\delta_n},k}^{b_i} \ast W_{{\delta_n},k}^{b_j} \right) (x) 
    + \Vreg^{\delta_n}(x ;b_i,b_j).
\end{equation}
Our key motivation for this splitting is that the convolution terms yield -- after summing in \eqref{En:sum} -- a \emph{non-negative} contribution to $E_n$ (see the derivation in \eqref{En:shortt}). Consequently, we obtain a lower bound on $E_n$ on compact subsets of $(\R^2)^n$ independent of $n$ and $\delta$.
\smallskip

\subsubsection{The $\Gamma$-limit of the interaction energy} 
\label{s:intro:results:G}
Our convergence result, Theorem \ref{t}, characterises the $\Gamma$-limit of the discrete energy $E_n$ with respect to a modified narrow topology. The narrow topology is natural when considering the limit of a collection of $n$ particles to a particle density. We modify it by requiring the particles to remain in a compact set. We choose this modification for the technical reason to obtain a lower bound on $E_n$, which is necessary for establishing a meaningful $\Gamma$-limit. Alternative choices would be to add a confining potential to $E_n$ or to consider a finite domain $\Omega$ with boundary conditions that keep the dislocations confined.

To describe this topology in detail, we first introduce the necessary framework. We fix an $n$-independent, finite set 
\begin{equation*} 
  \mathcal{B} := \{\xi_1,\dots,\xi_S\} \subset \mathbb S
\end{equation*}
of slip directions; see Figure \ref{fig:sett} for an example. We require that $b_i \in \mathcal{B}$ for any $i$ and any $n$. This requirement allows for a convenient relabelling of the dislocation positions: for each $s\in \{1,\dots,S\}$, we denote by $x_i^s \in \R^2$ with $i=1,\dots,n_s$ all the $n_s \in \N$ dislocations with Burgers vector $\xi_s$. We note that $\sum_{s=1}^S n_s = n$. Then, for each $\xi_s\in \mathcal{B}$, we consider the empirical measure
\begin{equation} \label{muns}
   \mu^s_n= \frac1n \sum_{i=1}^{n_s} \delta_{x_i^s} \in \mathcal M_+ (\R^2), \quad s=1,\dots, S,
\end{equation}
and observe that 
$$
  \bmu_n := (\mu_n^1,\dots,\mu_n^S) \in \mathcal P (\R^2 \times \{1, \ldots, S\}),
$$
where $\mathcal M_+ (\R^2)$ is the space of non-negative, finite measures on $\R^2$, and $\mathcal P (\R^2) \subset \mathcal M_+ (\R^2)$ is the subspace of  measures with unit mass, i.e., the space of probability measures. Finally, we rewrite $E_n$ as
\begin{align} \notag
  E_n (\bx; \bb) 
  &= \frac1{n^2} \sum_{s=1}^S \sum_{t=1}^S \sum_{i = 1}^{n_s} \sum_{j = 1}^{n_t} V_{\delta_n} (x_i^s - x_j^t; \xi_s, \xi_t) \\\label{En:mun:intro}
  &= \sum_{s=1}^S \sum_{t=1}^S \int_{\R^2} \int_{\R^2} V_{\delta_n} (x - y; \xi_s, \xi_t) \, d \mu_n^t (y) \, d \mu_n^s (x)
  =: E_n (\bmu_n).
\end{align}

Next we specify the modified narrow convergence for $\bmu_n$ in $\mathcal P (\R^2 \times \{1, \ldots, S\})$. Convergence in this topology means that $\mu_n^s$ converges in the narrow topology to some $\mu^s \in \mathcal M_+ (\R^2)$ as $n \to \infty$ for all $s \in \{1,\dots,S\}$, and that the support of $\mu_n^s$ is bounded uniformly in $n$ and $s$.

The $\Gamma$-limit of $E_n$ is an energy $E(\bmu)$ defined for $\bmu = (\mu^1, \dots, \mu^S) \in \mathcal P (\R^2 \times \{1, \ldots, S\})$. It has a somewhat complicated expression; see \eqref{E:def}. However, when $\bmu$ is absolutely continuous with sufficiently regular density, the expression simplifies to
\begin{align}\label{E:intro}
E (\bmu) 
= \sum_{s=1}^S \sum_{t=1}^S \int_{\R^2} \int_{\R^2} V (x - y; \xi_s, \xi_t) \, d \mu^t (y) \, d \mu^s (x),
\end{align}
which is clearly similar to \eqref{En:mun:intro}.

The modified narrow topology for $\bmu_n$ is however not the natural one for edge dislocations. Indeed, one can only observe the density of the net Burgers vector given by 
$$
   \kappa := \sum_{s=1}^S \xi_s \mu^s \in \mathcal M (\R^2;\R^2),
$$
which is in general not enough to reconstruct $\bmu$ uniquely. For this reason, in Corollary \ref{c:t} we study the $\Gamma$-convergence of $E_n$ with respect to the modified narrow convergence of 
\begin{equation} \label{kappan}
  \kappa_n := \sum_{s=1}^S \xi_s \mu_n^s \in \mathcal M (\R^2; \R^2)
\end{equation}
instead. The $\Gamma$-limit of $E_n$ is in this case a relaxation of the functional $E$ in \eqref{E:intro}; see \eqref{Em:def} for details.

\smallskip

In the general setting which we consider in Theorem \ref{t}, we replace in \eqref{En:mun:intro} and \eqref{E:intro} the integrands by some potentials $V_{\delta_n}^{st}(x-y)$ and $V^{st}(x-y)$, and interpret $\bmu$ as a list of densities of $S$ different species. Then, we interpret $V^{st}$ as the interaction potential between species $s$ and $t$, and $V_{\delta_n}^{st}$ as a regularisation of it. In Assumptions \ref{a:V} and \ref{a:Vd} we list the precise conditions on $V^{st}$ and $V_{\delta_n}^{st}$, which are strongly inspired by \eqref{V:W-form-intro} and \eqref{V:W-form-d-intro}. Note that this more general formulation requires no description of Burgers vectors $\xi_s$.

\subsection{Comments and outlook} 
\label{s:intro:disc}

We conclude the introduction by a discussion on the implications, applications and limitations of Theorem \ref{t}.
\smallskip

\emph{Special case in which regularisation is not needed}.
For the special case in which all potentials $V^{st}$ are lower-semicontinuous and bounded from below on bounded sets, the $\Gamma$-convergence result in Theorem \ref{t} applies to the energy $\tilde E_n$ in \eqref{Entil:unreg}. In this case, $V^{st}$ need not satisfy the structure assumption \eqref{V:W-form-intro}. This setting is typical for the single-species case, i.e., $S=1$. For $S = 1$, a $\Gamma$-convergence result has been established in \cite[Lem.~5.3]{CanizoPatacchini18DOI}.
\smallskip

\emph{Bounded domains}.
As mentioned at the start of Section \ref{s:intro:results:G}, an alternative choice for keeping the dislocations confined to a compact set is to consider a bounded domain $\Omega \subset \R^2$. However, care is needed, as the self-energy of dislocations may be unbounded from below in the vicinity of $\partial \Omega$ \cite{HudsonMorandotti17}. This problem is side-stepped in \cite{MoraPeletierScardia17} by requiring the dislocations to stay away from $\partial \Omega$ by a prefixed distance. Under this requirement the interaction of the dislocations with the boundary yields a continuous contribution to $E_n$, and therefore it causes no additional difficulty for proving a $\Gamma$-convergence result. However, enforcing the dislocations to stay a fixed distance away from $\partial \Omega$ is essentially the same as the modification of the topology employed in this paper.
\smallskip

\emph{Higher dimensions}.
While we have presented our main theorem in two spatial dimensions, our proof method applies to higher spatial dimensions with minor, obvious modifications.
\smallskip

\emph{Class of admissible regularisations}.
Since our main theorem, Theorem \ref{t}, allows for a class of regularisations (see Assumption \ref{a:Vd} for details and Section \ref{s:V:edge} for practical examples), it serves as a criterion for the available regularisations (see Section \ref{s:intro:Vreg}) under which a collection of edge dislocations can be approximated by a dislocation density. In particular, the related continuum energy is independent of the choice of regularisation. 

The assumption that the diagonal contribution $\gamma_n$ (see \eqref{gamman:bi}) vanishes as $n \to \infty$ is consistent with the viewpoint of zooming out as $n \to \infty$. Indeed, if we interpret $\delta_n$ as an atomic length scale and assume that $\delta_n$ scales linearly with the typical distance between neighouring dislocations (which is $1/\sqrt n$ in our two-dimensional setting), we obtain $\delta_n^2 \sim 1/n$. This is well within the range \eqref{deln:regime} which we obtained for a typical regularization $V_{\delta_n}$.
\smallskip

\emph{Class of admissible potentials}.
Theorem \ref{t} applies to a class of potentials $V^{st}$ (see Assumption \ref{a:V}) beyond \eqref{V:bibj}. In Section \ref{s:Riesz} we characterise  prototypical examples of potentials within this class, and demonstrate how to construct admissible regularisations thereof.

Here, we give a formal motivation for the key structure assumption in Section \ref{s:intro:results:VVd} from a mathematical viewpoint, and relate it to similar assumptions in the literature. Assuming that the Plancherel theorem applies, we obtain by applying the Fourier transform $\mathcal F f = \widehat f$ in \eqref{E:intro} that
\begin{align} \label{E:F-form}
E (\bmu) 
= \int_{\R^2} \sum_{s,t=1}^S \widehat{ V^{st} } \widehat{ \mu^t } \overline{ \widehat{ \mu^s } } \, d\omega.
\end{align}
A similar expression can be found for $E_n$. Then, to show that $E(\bmu)$ is bounded from below (non-negative in this case), it is sufficient to show that the matrix $(\widehat{ V^{st} } (\omega))_{s,t = 1}^S$ is positive semi-definite for all $\omega \in \R^2$. Since $(\widehat{ V^{st} } (\omega))_{s,t = 1}^S$ is in general not positive semi-definite, the role of $\Vreg$ in the decomposition in \eqref{V:W-form-intro} is to capture the non-positive semi-definite part. The role of the convolution terms is to make it obvious that the remainder is positive semi-definite; indeed,
\begin{equation*}
  \mathcal F ( V^{st} - \Vreg^{st} ) = \sum_{k=1}^K \overline{\widehat{ W_k^s }} \widehat{ W_k^t },
\end{equation*}
which clearly yields a non-negative contribution when inserted in \eqref{E:F-form}. 

Non-negativity or positivity of the Fourier transform of the interaction potential is a recurring assumption in the study of continuum interaction energies such as $E$ or in the study on the convergence of the underlying particle system; see, e.g., 
\cite{GarroniVanMeursPeletierScardia16,
GarroniVanMeursPeletierScardia19DOI,
GeersPeerlingsPeletierScardia13,
KimuraVanMeurs19acc,
vanMeurs18,
MoraRondiScardia19}. Here, we have decided to avoid using the Fourier transform in the assumptions on the interaction potentials to make the application to bounded domains easier. 
\smallskip

\emph{Comparison with \cite{GarroniLeoniPonsiglione10}}.
First we briefly recall the setting and the result of \cite{GarroniLeoniPonsiglione10} in a formal fashion. The energy considered is 
\begin{equation} \label{Ed:GLP10}
  E_\delta (\bx, \beta) = \int_{\Omega_\delta (\bx)} \C \beta : \beta,
\end{equation}
where $\bx$ is again the list of dislocation positions. The dependence of $n$ and $\delta_n$ is reversed such that the atomic length scale $\delta$ is the leading parameter. The domain $\Omega_\delta (\bx) := \Omega \setminus \cup_{i=1}^{n_\delta} B(x_i, \delta)$ is constructed by perforating a given domain $\Omega$ at each dislocation by removing a small ball $B(x_i, \delta)$ around it. The strain field $\beta : \Omega_\delta (\bx) \to \R^{2\times 2}$ is required to be compatible with each dislocation $(x_i, b_i)$ by the nonzero curl condition $\int_{ \partial B(x_i, \delta) } \beta \cdot ds = b_i / n^\delta$. Finally, the elasticity tensor $\C$ and the Frobenius inner product ``$:$" are defined in Section \ref{s:V:edge:aV}, and $E_\delta$ is the elastic energy of the perforated domain $\Omega_\delta (\bx)$.

The main result in \cite{GarroniLeoniPonsiglione10} is the $\Gamma$-convergence of $E_\delta$ as $\delta \to 0$ and $n_\delta \to \infty$ in a topology where $\kappa_{n_\delta}$ (as in \eqref{kappan}) converges to $\kappa$ in the narrow topology, and, roughly, $\beta_\delta \weakto \beta$ in $L^2$. The topology is further restricted by the geometrical constrains
\begin{equation}\label{GLP10:scaling}
  \frac1{ \sqrt{n_\delta} } \gg r_\delta \gg \delta^s \quad \forall \, s \in (0,1),
\qquad r_\delta := \min_{i \neq j} |x_i - x_j|,
\end{equation}
where $r_\delta$ is the minimal separation distance for the dislocations. The $\Gamma$-limit depends on the scaling regime of $n_\delta$ as $\delta \to 0$. There are three scaling regimes; $n_\delta \gg |\log \delta|$, $n_\delta \sim |\log \delta|$ and $n_\delta \ll |\log \delta|$. Here, we focus on the first regime, which is equivalent to the regime $\gamma_n \to 0$ in this paper. In this regime, the $\Gamma$-limit is
\[
  E_\delta (\beta) = \int_{\Omega_\delta (\bx)} \C \beta_s : \beta_s,
\]
where $\kappa = \Curl \beta$, and $\beta_s$ is the symmetric part of $\beta$.
\smallskip

Next we connect the setting and the result of \cite{GarroniLeoniPonsiglione10} with those of this paper. We expect that $\inf_\beta E_\delta(\bx, \beta)$ under the curl condition is an interacting particle energy of the form $E_n(\bx; \bb)$ as in \eqref{En:sum} with an additional term which describes the boundary effects due to the finite domain $\Omega$. This expectation is based on \cite{MoraPeletierScardia17}, where this is made precise in the case where all Burgers vectors $b_i$ are the same, and on Proposition \ref{p:V:CKK}, where we show that $V$ defined in \eqref{V:bibj} can be rewritten as an elastic energy of the form \eqref{Ed:GLP10}. Based on this connection between the discrete dislocation energies, we expect that the $\Gamma$-limits share a similar resemblance, i.e., that $\inf_\beta E(\beta)$ under the curl condition $\kappa = \Curl \beta$ resembles the relaxation of $E$ as obtained in Corollary \ref{c:t}. We have no rigorous prove for this.

Next we focus on the differences between the results in \cite{GarroniLeoniPonsiglione10} and those in this paper. The advantages of the result in \cite{GarroniLeoniPonsiglione10} are that the strain field $\beta$ is treated as a variable and that boundary effects are explicitly included. The advantages of the result in this paper is that no restrictions on the separation distance $r_\delta$ are imposed, that there is no upper bound on $n_\delta$ (\eqref{GLP10:scaling} implies essentially that $n_\delta$ has at most logarithmic growth), and that general regularisations of the dislocation cores are considered. Regarding the last advantage, we show in Section \ref{s:V:edge:perf} that a regularisation based on perforating the domain fits to our admissible class of regularisations, but that the equivalent of $E_\delta$ as in \eqref{Ed:GLP10} does not fit. This poses the question whether there is an alternative to the specific regularisation leading to $E_\delta$ for which no separation distance needs to be imposed.

More generally, the observation that the results in \cite{GarroniLeoniPonsiglione10} and those in this paper have several advantages over one another sparks the question on how both results can be combined to derive a $\Gamma$-limit in a general setting. In particular, keeping the strain field $\beta$ as a parameter and removing restrictions on the separation condition would lead to a rigorous justification of the elastic theory of continuously distributed dislocations (see \cite{Acharya01} and references therein), at least for static scenarios. For a specific physical example (on polygonisation) we refer to \cite[Sec.\ 4.2]{AroraAcharya20}. Since the proofs of both results in \cite{GarroniLeoniPonsiglione10} and in this paper are based on different techniques, the challenge on combining both approaches is beyond the scope of this paper, and left for future research.
\smallskip

\emph{Comparison with \cite{GarroniVanMeursPeletierScardia19DOI}}.
In \cite{GarroniVanMeursPeletierScardia19DOI}, evolutionary convergence is studied of the gradient flow of $E_n$ for edge dislocations with $S=2$ and $\mathcal{B} = \{e_1, -e_1\}$. The limiting gradient flow is that of the $\Gamma$-limit $E$. Evolutionary convergence is proven under the condition $\delta_n^2 \gg 1/\log n$. This condition is much more restrictive than ours in \eqref{deln:regime}, which is due to the dynamical setting. In addition, \cite{GarroniVanMeursPeletierScardia19DOI} provides a class of counterexamples for the convergence of the gradient flow of $E_n$ to that of $E$ under the condition $\delta_n \ll n^2$. Comparing this condition to \eqref{deln:regime}, we observe that there is quite a significant regime in the $(n, \delta)$ parameter space where the static convergence (i.e., $\Gamma$-convergence) holds, but where the evolutionary convergence fails. 

Next we translate this observation into a property of the regularised version of Volterra's model for edge dislocations (equipped with a linear drag law for their motion). During the gradient flow of $E_n$, \cite{GarroniVanMeursPeletierScardia19DOI} shows that dipoles (i.e., pairs of dislocations with opposite Burgers vector whose positions are very close) may impede the dynamics of $\bx$. Yet, the flow field in the gradient flow of $E$ is independent on any dipole density, and thus $\bmu$ cannot be impeded by a dipole density. The result of Theorem \ref{t} implies that an excess of dipoles can only lower the energy $E_n$ by a value which vanishes as $n \to \infty$. Hence, while the dynamical setting needs special treatment of dipoles, the static setting requires no such treatment.
\smallskip

\emph{The intriguing question on small $\delta_n$}.
Other than the {difference} in the scaling regimes of $\delta_n$ where Theorem \ref{t} and the evolutionary convergence in \cite{GarroniVanMeursPeletierScardia19DOI} hold, they have in common that both convergence results do not yield the expected interaction energy $E$ or the gradient flow thereof when $\delta_n$ is too small with respect to $n$. This is in line with a growing set of observations in recent literature (see, e.g., \cite{ChapmanXiangZhu15,vanMeurs18} and
\cite[Chap.~9]{VanMeurs15}) on continuum limits of multiple-species interacting particle systems with singular interactions. It is interesting to note that such interesting phenomena for small $\delta_n$ do not occur in single-species scenarios. Indeed, as mentioned above, it is often possible to prove such limits without the need to regularise the singularity.

Formally, the difficulty with small $\delta_n$ is that the diagonal terms in \eqref{En:sum} (described by $\gamma_n$ defined in \eqref{gamman:bi}) do not vanish as $n \to \infty$. This discrete effect is not captured in our proof of the $\Gamma$-liminf inequality in Theorem \ref{t}, which merely relies on the structure assumption in Section \ref{s:intro:results:VVd}. Hence, the resulting inequality is not sharp enough, and seems to require a renormalisation to recover the contribution of $\gamma_n$. We note that simply removing $\gamma_n$ from $E_n$ does not work, as such energy resembles $\tilde E_n$, for which there is no sufficient lower bound.

Despite this complexity, there are results in the literature where the case of small $\delta_n$ (i.e., $|\log \delta_n| \gtrsim n$) is treated. For instance, \cite{GarroniLeoniPonsiglione10} and \cite{MoraPeletierScardia17} consider this case under the restriction that dislocations remain separated by a separation distance  $r_n$. This results effectively in a further regularisation of $V_{\delta_n}$ at scale $r_n$. In more recent results, however, this separation condition has been removed (see \cite{DeLucaGarroniPonsiglione12} and \cite{Ginster19}) by means of a \emph{ball-construction} in the Ginzburg-Landau spirit. This ball-construction relies on the detailed structure of the corresponding elastic energy on the perforated domain. Hence, for any other kind of regularisation considered in this paper, it remains an open problem to obtain a meaningful limit as $n \to \infty$ of a rescaled or renormalised version of $E_n$ when $\gamma_n$ does not vanish as $n \to \infty$. 
\medskip 
 
The paper is organised as follows. In Section \ref{s:t} we state and prove our main result, Theorem \ref{t}, on the $\Gamma$-convergence of $E_n$ for a class of interaction potentials. In Corollary \ref{c:t} we prove the $\Gamma$-convergence for a different topology on $\kappa_n$ in \eqref{kappan}. In Section \ref{s:V:edge} we show that Theorem \ref{t} applies to the case of edge dislocations, and that several of the regularisations used in the literature (such as mollification and a type of perforation) fit to the assumptions on $V_{\delta_n}$. Preceding this section, in Section \ref{s:Riesz} we show how Theorem \ref{t} applies to other potentials of interest, whose treatment allows for lighter computations.

\section{The $\Gamma$-limit in the general setting}
\label{s:t}

In this section we prove our main $\Gamma$-convergence result, Theorem \ref{t}, for a certain class of potentials. We use the same notation as in Section \ref{s:intro:results:G}, except for the interaction potential between two particles of species $s, t \in \{1, \dots, S\}$, which we denote by $V^{st}$. This covers the case of edge dislocations by setting $V^{st}(x) = V(x; \xi_s,\xi_t)$. Similarly, the regularised interaction potential is denoted as $V^{st}_{\delta}$. Then, the energy for the particle system is given by
\begin{equation*} 
  E_n (\bmu_n) 
  = \sum_{s=1}^S \sum_{t=1}^S \int_{\R^2} \int_{\R^2} V_{\delta_n}^{st} (x-y) \, d \mu_n^t (y) \, d \mu_n^s (x),
\end{equation*} 
which is the analogue of \eqref{En:mun:intro}. Again, we consider $\delta_n \to 0$ as $n \to \infty$, and assume that
\begin{equation} \label{gamman:ss}
  \gamma_n = \frac1{n^2} \sum_{s=1}^S \sum_{i=1}^{n_s} \big| V_{\delta_n}^{ss} (0) \big|
\end{equation}
tends to $0$ as $n \to \infty$. The assumptions on $V^{st}$ and $V^{st}_{\delta}$ are specified in the following.

\paragraph{Assumptions on $V^{st}$ and $V^{st}_{\delta_n}$}
 
Here we translate the properties \eqref{V:W-form-intro} and \eqref{V:W-form-d-intro} of the potentials for edge dislocations in terms of precise assumptions on the general potentials $V^{st}$ and $V_\delta^{st}$.

\begin{ass}[Properties of $V^{st}$] \label{a:V} 
We assume that $V^{st}$ decomposes as
\begin{equation} \label{V:W-form}
  V^{st} (x) = \sum_{k=1}^{\mathsf K} \left( \overline W_k^s \ast W_k^t \right) (x) + \Vreg^{st} (x)
  \quad \text{for a.e.\ } x \in \R^2,
\end{equation}
where $\mathsf K \in \N$ and $\overline W_k^s (x) := W_k^s(-x)$. The potentials $\Vreg^{st}$ and $W_k^s$ are such that, for all $s,t,k$,
\begin{enumerate}[label=(\roman*)]
  \item $\Vreg^{st} \in C (\R^2)$;
  \item $W_k^s \in L^1 (\R^2)$;
  \item $V^{st} \in C (\R^2 \setminus \{0\})$ is even.
\end{enumerate}
\end{ass}

Given a potential $V^{st}$ which satisfies Assumption \ref{a:V}, we impose the following (minimal) requirements on the regularised potential $V_{\delta}^{st}$.

\begin{ass}[Properties of $V_\delta^{st}$] \label{a:Vd} 
For any $s,t,k$, let $V^{st}$, $\Vreg^{st}$ and $W_k^s$ be as in Assumption \ref{a:V}. Then, for any $\delta > 0$, the potential $V_\delta^{st}$ can be decomposed as
\begin{equation} \label{V:W-form-d}
  V_\delta^{st} (x) = \sum_{k=1}^{\mathsf K} \left( \overline W_{\delta, k}^s * W_{\delta, k}^t \right)(x) + \Vreg^{\delta, st} (x)
  \quad \text{for a.e.\ } x \in \R^2,
\end{equation}
where $V_\delta^{st}$, $\Vreg^{\delta, st}$ and $W_{\delta, k}^s$ are such that:
\begin{enumerate}[label=(\roman*)]
\item \label{a:Vd:reg:even} (Regularity). For all $\delta > 0$, the function $V_\delta^{st} \in C (\R^{2})$ is even, $W_{\delta, k}^s \in L^1(\R^2)$ and $\Vreg^{\delta, st} \in C (\R^{2})$;
\item \label{a:Vd:uf:conv} (Uniform convergence in any annulus). \\ For all $0 < \varepsilon < 1$, there holds $V_\delta^{st} \to V^{st}$ in $C \big( B(0, 1/\varepsilon) \setminus  B(0, \varepsilon) \big)$ as $\delta \to 0$;
\item \label{a:Vd:W-form} (Convergence of $\Vreg^{\delta, st}$ and $W_{\delta, k}^s$). \\
(a) \ $W_{\delta, k}^{s} * \varphi \to W_{k}^s * \varphi$ in $C_{\text{loc}} (\R^2)$ as $\delta\to 0$ for any $\varphi \in C_c^\infty (\R^2)$, \smallskip \\ 
(b) \ $\Vreg^{\delta, st} \to \Vreg^{st}$ in $C_{\text{loc}} (\R^2)$ as $\delta\to 0$; 
\item \label{a:Vd:dom} (Dominator). There exists a radially symmetric $U^{st}$ in $L^1_{\text{loc}}(\R^2) \cap C(\R^2 \setminus \{0\})$, non-increasing in the radial direction, such that $U^{st}(x) \geq \sup_{0 < \delta < 1} |V_\delta^{st}(x)|$ for all $x\in B(0,1)$.
\end{enumerate}
\end{ass}

\medskip
 
Next we comment on the motivation for the four conditions in Assumption \ref{a:Vd}. Condition \ref{a:Vd:reg:even} ensures sufficient regularity, \ref{a:Vd:uf:conv} ensures consistency with $V^{st}$, and \ref{a:Vd:W-form} assumes that the components of the decomposition are consistent with those in \eqref{V:W-form}. Condition \ref{a:Vd:dom} is imposed for technical reasons; it provides a uniform upper bound on $V_\delta^{st}$ close to $0$. We rely on this condition when constructing a recovery sequence in the proof of Theorem \ref{t}. Even when $V_\delta$ does not satisfy this condition, or when this condition is difficult to prove, Theorem \ref{t} may still hold. In such case, only the construction of the recovery sequence needs to be redone.

To illustrate that Assumptions \ref{a:V} and \ref{a:Vd} cover several potentials of interest, including that of edge dislocations, we show in Sections \ref{s:Riesz} and \ref{s:V:edge} how such potentials can be shown to satisfy these assumptions. 

\paragraph{Rewriting the energy}

As mentioned in the introduction, the decomposition of $V_\delta^{st}$ in \eqref{V:W-form-d} provides a uniform lower bound on the energy, at least when $\bmu_n$ is compactly supported. This can be seen by simply rewriting the energy as follows:
\begin{align} \notag
  E_n (\bmu_n) 
  &= \sum_{s,t=1}^S \int_{\R^2} \int_{\R^2} V_{\delta_n}^{st} (x-y) \, d \mu_n^t (y) \, d \mu_n^s (x)
  = \sum_{s,t=1}^S \int_{\R^2} \big( V_{\delta_n}^{st} * \mu_n^t \big) \, d \mu_n^s \\\notag
  &= \sum_{s,t=1}^S \bigg[ \int_{\R^2} \big( \Vreg^{\delta_n,st} * \mu_n^t \big) \, d \mu_n^s + \sum_{k=1}^{\mathsf K} \int_{\R^2} \big( \overline W_{\delta_n, k}^s * (W_{\delta_n, k}^t * \mu_n^t ) \big) \, d \mu_n^s  \bigg] \\\label{En:shortt}
  &= \sum_{s,t=1}^S \int_{\R^2} \big( \Vreg^{\delta_n, st} * \mu_n^t \big) \, d \mu_n^s + \sum_{k=1}^{\mathsf K} \bigg\| \sum_{s=1}^S W_{\delta_n, k}^s * \mu_n^s \bigg\|_{L^2(\R^2)}^2.
\end{align}
Indeed, the second term in the right-hand side is non-negative, and the first term is bounded from below because of Assumption \ref{a:Vd}\ref{a:Vd:uf:conv} and the imposed bound on the support of $\bmu_n$. To ensure that the support of $\bmu_n$ remains uniformly bounded in our $\Gamma$-convergence result, we modify the narrow topology by requiring the support to be bounded.

\paragraph{The modified narrow topology} 

We define the modified narrow topology on $\mathcal P ( \R^2 \times \{1, \ldots, S\} )$ as follows: 
\begin{equation} \label{cTop}
  \bmu_n \xweakto c \bmu \text{ as } n \to \infty 
  \quad \Longleftrightarrow \quad
  \left\{ \begin{aligned}
        &\mu_n^s \weakto \mu^s \text{ as $n \to \infty$ for all $s$, and} \\
        &\bigcup_{n = 1}^\infty \bigcup_{s=1}^S \supp \mu_n^s \: \text{ is bounded.} 
  \end{aligned} \right.
\end{equation}
Here, $\mu_n^s \weakto \mu^s$ denotes the narrow convergence of $\mu_n^s$ to $\mu^s$ in $\mathcal M_+(\R^2)$. It is defined as follows:
\begin{equation*}
  \mu_n^s \weakto \mu^s \text{ as } n \to \infty 
  \quad \Longleftrightarrow \quad
  \forall \, \varphi \in C_b(\R^2) : 
    \int_{\R^2} \varphi(x) \, d\mu_n^s(x) \xto{n \to \infty} \int_{\R^2} \varphi(x) \, d\mu^s(x).
\end{equation*}

\paragraph{The main $\Gamma$-convergence result}

Before stating the main result, we give the precise definitions of $E_n$ and $E$:
\begin{subequations} \label{En:def}
\begin{align} \label{En:dom}
  D ( E_n )  
  &:= \big\{ \bmu \in \mathcal P ( \R^2 \times \{1, \ldots, S\} ) : \exists \, x_i^s \in \R^2 : \bmu \text{ satisfies \eqref{muns}} \big\}, \\\label{En:mun} 
  E_n (\bmu_n)
  &= \sum_{s,t=1}^S \int_{\R^2} \big( \Vreg^{\delta_n, st} * \mu_n^t \big) \, d \mu_n^s 
    + \sum_{k=1}^{\mathsf K} \bigg\| \sum_{s=1}^S W_{\delta_n, k}^s * \mu_n^s \bigg\|_{L^2(\R^2)}^2,
\end{align}
\end{subequations}
\begin{subequations} \label{E:def}
\begin{align} \label{E:dom}
  D(E) &:= \big\{ \bmu \in \mathcal P (\R^2 \times \{1, \ldots, S\}) : \supp \bmu \text{ bounded} \big\}, \\\label{E:short}
  E (\bmu) 
  &= \sum_{s,t=1}^S \int_{\R^2} \big( \Vreg^{st} * \mu_n^t \big) \, d \mu_n^s + \sum_{k=1}^{\mathsf K} \bigg\| \sum_{s=1}^S W_k^s * \mu^s \bigg\|_{L^2(\R^2)}^2.
\end{align} 
\end{subequations}
The expression for $E_n$ is taken from \eqref{En:shortt}. The computations in \eqref{En:shortt} and \eqref{En:mun:intro} show how it can be rewritten explicitly as a particle interaction energy. Regarding $E$, we note that it is well-defined on $D(E)$ with values in $(-\infty, \infty]$. Indeed, the first term is finite because $\supp \bmu$ is bounded. In the second term, $W_k^s * \mu^s$ can be regarded as a distribution, for which the $L^2$-norm is well-defined with values in $[0, \infty]$. In particular, if $\bmu$ is absolutely continuous with density in $L^2$, then we can follow the computation in \eqref{En:shortt} in reverse order to rewrite $E(\bmu)$ as in \eqref{E:intro}.

\begin{thm}[$\Gamma$-convergence] \label{t} 
Let $S \in \N{}$ be positive. For any $s,t \in \{1,\ldots,S\}$, let $V^{st}$ satisfy Assumption \ref{a:V} for some $\mathsf K$, $\Vreg^{st}$ and $W_k^s$. Let $\delta_n \to 0$ as $n \to \infty$, and let $V_{\delta_n}^{st}$ for each integer $n \geq 1$ satisfy Assumption \ref{a:Vd} for some $\Vreg^{\delta_n, st}$ and $W_{\delta_n, k}^s$. Let $E_n$ and $E$ be the energies defined respectively in \eqref{En:def} and \eqref{E:def}. If $\gamma_n \to 0$ as $n \to \infty$ (see \eqref{gamman:ss}), then $E_n$ $\Gamma$-converges to $E$ with respect to the modified narrow topology \eqref{cTop}, i.e.,
\begin{subequations}
\label{tf:Gconv}
\begin{alignat}2
\label{tf:Glimf}
&\forall \, \bmu \in D(E) \ \forall \, \bmu_n \in D(E_n), \, \bmu_n \xweakto c \bmu : 
& \liminf_{n\to\infty}  E_n (\bmu_n) 
&\geq E(\bmu), 
\\
&\forall \, \bmu \in D(E) \ \exists \, \bmu_n \in D(E_n), \, \bmu_n \xweakto c \bmu : 
& \ \limsup_{n\to\infty} E_n (\bmu_n) 
&\leq E(\bmu).
\label{tf:Glimp}
\end{alignat}
\end{subequations} 
\end{thm}

\begin{proof}
For convenience, we set $\widetilde V^{st} : \R^2 \times \R^2 \to \overline{\R}$ as $\widetilde V^{st} (x,y) := V^{st} (x-y)$, and use the same notation for other potentials. We write $E_n = G_n + F_n$ with
\begin{equation} \label{En:short}
  G_n (\bmu_n)
  := \sum_{s,t=1}^S \iint_{\R^2 \times \R^2} \Vregt^{\delta_n, st} \, d( \mu_n^s \otimes \mu_n^t ) 
  \quad \text{and} \quad
  F_n (\bmu_n)
  := \sum_{k=1}^{\mathsf K} \bigg\| \sum_{s=1}^S W_{\delta_n, k}^s * \mu_n^s \bigg\|_{L^2(\R^2)}^2.
\end{equation}
Analogously, we write $E = G+F$.

First we prove that $G_n$ is a continuous perturbation with respect to the modified narrow topology, i.e., for all $\bmu \in D(E)$ and all $\bmu_n \in D(E_n)$ with $\bmu_n \xweakto c \bmu$, it holds that $G_n (\bmu_n) \to G (\bmu)$ as $n \to \infty$.
With this aim, we fix $s,t$, and let $\bmu$ and $\bmu_n$ be arbitrary such that $\bmu_n \xweakto c \bmu$ as $n \to \infty$. Let $K \subset \R^2$ be the compact set which satisfies $\supp \bmu_n \subset K$ for all $n \in \Np{}$. By Assumption \ref{a:Vd}\ref{a:Vd:W-form}$(b)$ it holds that $\Vregt^{\delta_n, st} \to \Vregt^{st}$ uniformly in $K \times K$, as $n \to \infty$. Since $\mu_n^s \weakto  \mu^s$ and $\mu_n^t \weakto  \mu^t$ in $\mathcal M_+ (K)$ as $n \to \infty$, it also holds that $\mu_n^s \otimes \mu_n^t \weakto  \mu^s \otimes \mu^t$ in $\mathcal M_+ (K \times K)$, and thus
\begin{equation*}
  \lim_{n \to \infty} G_n(\bmu_n)
  = \lim_{n \to \infty} \iint_{\R^2 \times \R^2} \Vregt^{\delta_n, st} \, d( \mu_n^s \otimes \mu_n^t )
  = \iint_{\R^2 \times \R^2} \Vregt^{st} \, d( \mu^s \otimes \mu^t )
  = G(\bmu).
\end{equation*}
Hence, it remains to show that $F_n$ $\Gamma$-converges to $F$ with respect to the modified narrow topology.
\smallskip

\emph{The liminf inequality \eqref{tf:Glimf} for $F_n$}. Let $\bmu_n \xweakto c \bmu$ be such that $F_n (\bmu_n)$ is bounded. Then, from \eqref{En:short} we obtain that $\sum_{s=1}^S W_{\delta_n, k}^s * \mu_n^s$ is bounded in $L^2 (\R^2)$ for all $k$. Hence, it converges along a subsequence (not relabelled) weakly in $L^2 (\R^2)$ to some $f_k$. We characterise $f_k$ by computing the distributional limit of $\sum_{s=1}^S W_{\delta_n, k}^s * \mu_n^s$ as $n \to \infty$. For any $\varphi \in C_c^\infty (\R^2)$, it follows from Assumption \ref{a:Vd}\ref{a:Vd:W-form}$(a)$ that
\begin{align*}
  \left\langle \sum_{s=1}^S W_{\delta_n, k}^s * \mu_n^s, \varphi \right\rangle
  = \sum_{s=1}^S \left\langle \mu_n^s, \overline W_{\delta_n, k}^s * \varphi \right\rangle
  \xto{n \to \infty} \sum_{s=1}^S \left\langle \mu^s, \overline W_k^s * \varphi \right\rangle
  = \left\langle \sum_{s=1}^S W_k^s *  \mu^s, \varphi \right\rangle,
\end{align*}
where $\langle \cdot \, , \cdot \rangle$ is the duality pairing consistent with the inner product in $L^2 (\R^2)$.
This implies that $f_k = \sum_{s=1}^S W_k^s * \mu^s$, which is independent of the choice of the subsequence. We conclude that
\begin{align*}
  \liminf_{n \to \infty} F_n (\bmu_n) 
  \geq \sum_{k=1}^{\mathsf K} \liminf_{n \to \infty} \bigg\| \sum_{s=1}^S W_{\delta_n, k}^s * \mu_n^s \bigg\|_{L^2(\R^2)}^2
  \geq \sum_{k=1}^{\mathsf K} \bigg\| \sum_{s=1}^S W_{k}^s * \mu^s \bigg\|_{L^2(\R^2)}^2 = F(\bmu). 
\end{align*}
\smallskip

\emph{The limsup inequality \eqref{tf:Glimp} for $F_n$}. We first show, by a density argument, that it is enough to establish \eqref{tf:Glimp} for $\bmu \in L^\infty (\R^2 \times \{1, \ldots,S\}) \cap D(E)$. With this aim, let $\bmu \in D(E)$, and set $\mu_\varepsilon^s := \eta_\varepsilon * \mu^s$ with $\eta_\varepsilon$ the usual mollifier. Clearly, $\bmu_\varepsilon \in C_c^\infty (\R^2 \times \{1, \ldots,S\}) \cap D(E)$ and $\bmu_\varepsilon \xweakto c \bmu$ as $\e \to 0$. Since $F(\bmu)$ is finite, we have $\sum_{s=1}^S W_k^s * \mu^s \in L^2(\R^2)$. Hence
\begin{align*} 
  \sum_{s=1}^S W_k^s * \mu_\varepsilon^s 
  = \eta_\varepsilon * \bigg( \sum_{s=1}^S W_k^s * \mu^s \bigg) 
  \xrightarrow{\varepsilon \to 0} \sum_{s=1}^S W_k^s * \mu^s 
  \quad \text{in } L^2 (\R^2)
\end{align*}
for every $k$, and thus $F (\bmu_\varepsilon ) \to F (\bmu)$ as $\e \to 0$. 

Next we fix any $\bmu \in L^\infty (\R^2 \times \{1, \ldots,S\}) \cap D(E)$ and construct a recovery sequence $\bmu_n \in D(E_n)$ for it by `discretising' the density $\bmu$. We do this by putting the particles $x_i^s$ on a square lattice to guarantee their separation.  

As a first step, we define for any $n \geq 1$ the lattice $\Lambda_n := r_n \Z^2$ with
\begin{equation*}
  r_n := \frac 1 { \lceil \sqrt S \rceil \sqrt{ n \| \bmu \|_\infty } }
\end{equation*}
and $S$ coarser sub-lattices thereof given by $\Lambda_n^s := \lceil \sqrt S \rceil \Lambda_n + \{\ell_s r_n\}$ for $s = 1, \ldots, S$ for some $\ell_s \in \{0, 1, \ldots, \lceil \sqrt S \rceil - 1 \}^2$, where $\ell_1, \ldots , \ell_S$ are all different. We note that $\Lambda_n^s$ and $\Lambda_n^t$ are disjoint whenever $s \neq t$. 

Then, for each species $s$, we approximate $\mu^s$ by choosing $n_s$ points $x_i^s \in \Lambda_n^s$ such that 
\begin{itemize}
  \item $\dist (x_i^s, \supp \mu^s) \leq 1/\| \bmu \|_\infty$ for all $i$,
  \item $x_i^s \neq x_j^s$ for all $i \neq j$,
\end{itemize}
and such that the corresponding measure $\bmu_n$ defined in \eqref{muns} satisfies $\bmu_n \in D(E_n)$ (i.e., $\sum_{s=1}^S n_s = n$) and $\bmu_n \xweakto c \bmu$ as $n \to \infty$. Then, the particle positions satisfy the following separation condition: 
\begin{equation} \label{for:pf:limp:Enpm:1}
  \forall \, n \geq 1: 
  \min \big\{ \big| x_i^s - x_j^t \big| : i \neq j \text{ or } s \neq t \big\}
  \geq r_n.
\end{equation}
In addition, it is easy to check that $\supp \bmu_n \cup \supp \bmu \subset K$ for some compact set $K \subset \R^2$ independent of $n$.

Finally, we show that $\bmu_n$ is a recovery sequence for $\bmu$, i.e.,  $E_n(\bmu_n) \to E (\bmu)$ as $n\to \infty$. Since $\bmu \in L^\infty$, we recall that $E (\bmu)$ can be recast as in \eqref{E:intro}, and thus it suffices to show that
\begin{equation} \label{fp:limp:red}
  \iint_{\R^2 \times \R^2} \widetilde V_{\delta_n}^{st} \, d( \mu_n^s \otimes \mu_n^t )
  \xto{n \to \infty}
  \iint_{\R^2 \times \R^2} \widetilde V^{st} \, d( \mu^s \otimes \mu^t )
  \quad \text{for all } s, t \in \{1, \ldots, S\}.
\end{equation}

To prove \eqref{fp:limp:red}, we take $0 < \e < 1$ arbitrary, and define the open `thick diagonal' in $\R^2 \times \R^2$ as
\begin{equation*}
  \Delta_\e := \{ (x,y) \in \R^2 \times \R^2 : |x - y| < \e \}.
\end{equation*}
Since $\bmu$ is absolutely continuous with bounded density, $\mu_n^s \otimes \mu_n^t \xweakto c \mu^s \otimes \mu^t$ on $\mathcal M_+ (\Delta_\e^c)$ as $n \to \infty$. Then, by Assumption \ref{a:Vd}\ref{a:Vd:uf:conv} we have that 
\begin{equation*} 
  \iint_{\Delta_\e^c} \widetilde V_{\delta_n}^{st} \, d( \mu_n^s \otimes \mu_n^t )
  \xto{n \to \infty}
  \iint_{\Delta_\e^c} \widetilde V^{st} \, d( \mu^s \otimes \mu^t )
  \quad \text{for all } s, t \in \{1, \ldots, S\}.
\end{equation*}
To conclude \eqref{fp:limp:red}, it suffices to show that
\begin{equation} \label{fp:limp:m-est}
  \bigg| \iint_{\Delta_\e} \widetilde V^{st} \, d( \mu^s \otimes \mu^t ) \bigg|
  + \bigg| \iint_{\Delta_\e} \widetilde V_{\delta_n}^{st} \, d( \mu_n^s \otimes \mu_n^t ) \bigg|
  \leq c_\e
  \quad \text{for all } s, t \in \{1, \ldots, S\},
\end{equation}
where the constant $c_\e > 0$ is independent of $n$, $s$ and $t$, and $c_\e  = o(1)$ as $\e \to 0$. 

For the first integral in \eqref{fp:limp:m-est}, we estimate
\begin{multline*}
  \bigg| \iint_{\Delta_\e} \widetilde V^{st} \, d( \mu^s \otimes \mu^t ) \bigg|
  \leq \int_{\R^2} \int_{B(x,\e)} \big| V^{st}(x-y) \big| \, d \mu^t(y) \, d \mu^s(x) \\
  \leq \mu^s(\R^2) \| \mu^t \|_\infty \int_{B(0,\e)} |V^{st}|
  \leq \| \bmu \|_\infty \int_{B(0,\e)} |V^{st}|,
\end{multline*}
which is $o(1)$ as $\e \to 0$ since $V^{st} \in L^1_{\text{loc}} (\R^2)$. 

To estimate the second integral in \eqref{fp:limp:m-est}, we first treat the case $s \neq t$. Then, similar to the computation above, we obtain from Assumptions \ref{a:Vd}\ref{a:Vd:dom} that
\begin{multline} \label{fp:est:4}
  \bigg| \iint_{\Delta_\e} \widetilde V_{\delta_n}^{st} \, d( \mu_n^s \otimes \mu_n^t ) \bigg|
  \leq \int_{\R^2} \int_{B(x,\e)} \big| V_{\delta_n}^{st}(x-y) \big| \, d \mu_n^t(y) \, d \mu_n^s(x) \\
  \leq \frac1{n^2} \sum_{i=1}^{n_s} \sum_{| x_i^s - x_j^t | < \varepsilon} U^{st} (x_i^s - x_j^t) =: \beta_\e^n.
\end{multline}
We interpret $\beta_\e^n$ as a discretization of $\beta_\e := \iint_{\Delta_\e} \widetilde U^{st} \, d( \mu^s \otimes \mu^t )$, which is $o(1)$ as $\e \to 0$ and independent of $n$. However, $\beta_\e^n$ may be larger than $\beta_\e$ because of the singularity of $U^{st}$. Yet, by the construction of $\bmu_n$ (see, e.g., $\Lambda_n$, $\Lambda_n^s$ and $r_n$ defined above), it is not difficult to construct an $n$-independent constant $C > 0$ such that $\beta_\e^n \leq C \beta_\e$. We finish the proof of \eqref{fp:limp:red} for $s \neq t$ with a sketch of a possible construction for this constant $C$.

First, by assuming that all lattice sites of $\Lambda_n^t$ are occupied, we obtain
\begin{equation} \label{pf:m:est}
  \beta_\e^n
  \leq \frac1{n^2} \sum_{i=1}^{n_s} \sum_{y \in \Lambda_n^t \cap B (x_i^s, \varepsilon) } U^{st} (x_i^s - y)
  = \frac1{n} \sum_{y \in \Lambda_n^t \cap B (x_1^s, \varepsilon) }  U^{st} (x_1^s - y),
\end{equation}
where we have used that $\Lambda_n^s$ is periodic for any $s$. Then, we order the points in $\Lambda := \Lambda_n^t \cap B (x_1^s, \varepsilon) = \{y_1, y_2, \ldots, y_K\}$ by their distance from $x_1^s$, starting with the closest. At the same time, we tile $B (x_1^s, \varepsilon)$ with annuli $A_1, A_2, \ldots, A_K$ centred at $x_1^s$ with fixed area $a/n > 0$ such that $A_k$ is enclosed by $A_{k+1}$ and $y_k \notin \cup_{\ell=1}^k \overline A_k$.

Next we sketch the argument that such $a > 0$ can be constructed independently of $n$ and $\e$. For the first four points $y_1, \dots, y_4$, we obtain from   \eqref{for:pf:limp:Enpm:1} that $|x_1^s - y_k| \geq r_n$, and thus it is sufficient to have that $\cup_{\ell=1}^4 A_k \subset B(x_1^s, r_n)$. This yields the following condition on $a$:
\begin{equation*}
   4 \frac an = \sum_{k=1}^4 |A_k| \leq \pi r_n^2 = \frac\pi{ \lceil \sqrt S \rceil^2 \| \bmu \|_\infty n }.
 \end{equation*} 
 For the next 12 points $y_5, \dots, y_{16}$, we have by construction that $|x_1^s - y_k| \geq \lceil \sqrt S \rceil r_n$. Hence, it is sufficient to have that $\cup_{\ell=1}^{16} A_k \subset B(x_1^s, \lceil \sqrt S \rceil r_n)$. This yields the following condition on $a$:
\begin{equation*}
   16 \frac an \leq \pi \lceil \sqrt S \rceil^2 r_n^2 = \frac\pi{ \| \bmu \|_\infty n }.
 \end{equation*} 
Continuing this construction for points $y_k$ that are at least 
$$2 \lceil \sqrt S \rceil r_n, \, 3 \lceil \sqrt S \rceil r_n, \dots$$ 
far away from $x_1^s$, we obtain the sufficient condition
$
   a \leq  c / \| \bmu \|_\infty
$
for some $c > 0$ independent of $\e$ and $n$.

Finally, with $a$ and $A_k$ as constructed above, we use that $U^{st}$ is decreasing in the radial direction to continue \eqref{pf:m:est} by 
\begin{multline*}
  \beta_\e^n 
  \leq \frac1{n} \sum_{k=1}^K  U^{st} (x_1^s - y_k)
  \leq \frac1{n} \sum_{k=1}^K \frac1{|A_k|} \int_{A_k} U^{st} (x_1^s - y) \, dy \\
  = \frac1a \sum_{k=1}^K \int_{A_k} U^{st} (x_1^s - y) \, dy 
  \leq \frac1a \int_{B (0, \varepsilon)} U^{st}
  = \frac{\beta_\e}a.
\end{multline*}
This completes the proof of \eqref{fp:limp:m-est} for the case $s \neq t$.

The case $s = t$ can be treated with a minor adjustment. Indeed, in \eqref{fp:est:4} the self-interaction of $x_i^s$ needs to be separated from the summation before estimating $V_{\delta_n}^{ss}$ by $U^{ss}$. This yields
\begin{equation*} 
  \bigg| \iint_{\Delta_\e} \widetilde V_{\delta_n}^{ss} \, d( \mu_n^s \otimes \mu_n^s ) \bigg|
  \leq \frac1n \big| V_{\delta_n}^{ss}(0) \big| +  \frac1{n^2} \sum_{i=1}^{n_s} \sum_{0 < | x_i^s - x_j^s | < \varepsilon} U^{ss} (x_i^s - x_j^s).
\end{equation*}
The second term can be treated analogously to the case $s \neq t$; the first term is bounded by $\gamma_n$, which is assumed to vanish as $n \to \infty$.
\end{proof}


%


\paragraph{$\Gamma$-convergence in a weaker topology}
As mentioned in the introduction, for the case of edge dislocations, it is physically more relevant to prove a $\Gamma$-convergence result for the topology on $\kappa_n$ as defined in \eqref{kappan}. We do this here for a general potential $V^{st}$ satisfying Assumption \ref{a:V}, and use the set of Burgers vectors $\mathcal B = \{\xi_1, \ldots, \xi_S\}$ only to define the topology in \eqref{kappan}.

To apply Theorem \ref{t} to $E_n$ when written as a function of $\kappa_n$, it is convenient to be able to reconstruct $\bmu_n$ uniquely from $\kappa_n$. A sufficient condition for this reconstruction is that no two dislocations are at the same position. Hence, we set
\begin{equation} \label{En:dom:circ}
  D_\circ (E_n) := \Big\{ \kappa_n = \frac1n \sum_{s=1}^S \sum_{i=1}^{n_s} \xi_s \delta_{x_i^s} : \min_{(s,i) \neq (t,j)} |x_i^s - x_j^t | > 0 \Big\}.
\end{equation}
Then, on $D_\circ (E_n)$, $E_n$ can be defined as in \eqref{En:mun} as a function of $\kappa_n$.

In the continuum setting it is not always possible to reconstruct $\bmu$ uniquely from 
\begin{equation} \label{kappa}
   \kappa = \sum_{s=1}^S \xi_s \mu^s.
\end{equation}
Indeed, if $\mathcal B$ is a set of linearly dependent vectors, then $|\kappa|(\R^2)$ can attain any value in $[0,1]$. In particular, if $|\kappa|(\R^2) < 1$, then there is no unique $\bmu$ for which \eqref{kappa} holds. For $\kappa \in \mathcal M (\R^2; \R^2)$, let
\begin{equation*}
  \mathcal A(\kappa)
  := \bigg\{ \bmu \in D(E) : \kappa = \sum_{s=1}^S \xi_s \mu^s \text{ and }  \supp \bmu \subset \{0\} \cup \supp \kappa  \bigg\}
\end{equation*}
be the related set of compatible $\bmu$. The singleton $\{0\}$ is added to the support with the sole purpose to have $\mathcal A(0)$ non-empty. Since $\mathcal A(\kappa)$ may be empty for other $\kappa \in \mathcal M (\R^2; \R^2)$, we set
\begin{subequations} \label{Em:def}
\begin{align} \label{Em:dom}
  D(\mathcal E) 
  := \big\{ \kappa \in \mathcal M (\R^2; \R^2) : \mathcal A(\kappa) \neq \emptyset \big\}
\end{align}
as the support of the continuum limit $\mathcal E$, which is then defined as the relaxation of $E$ defined in \eqref{E:def} over $\mathcal A(\kappa)$:
\begin{align} \label{Em:short}
  \mathcal E (\kappa) 
  := \inf_{\bmu \in \mathcal A(\kappa)} E (\bmu). 
\end{align}
\end{subequations}

For our $\Gamma$-convergence result, we use again the modified topology, but this time on the space $\mathcal M (\R^2; \R^2)$:
\begin{equation} \label{cTop:R2}
  \kappa_n \xweakto c \kappa \text{ as } n \to \infty 
  \quad \Longleftrightarrow \quad
  \left\{ \begin{aligned}
        &\kappa_n \weakto \kappa \text{ as $n \to \infty$, and} \\
        &\bigcup_{n = 1}^\infty \supp \kappa_n \: \text{ is bounded.} 
  \end{aligned} \right.
\end{equation}

\begin{cor}[$\Gamma$-convergence] \label{c:t}
Let $S \in \N$ be positive, and let $\xi_1, \dots, \xi_S \in \mathbb S$. Let $E_n$ be as in \eqref{En:mun:intro} with domain $D_\circ (E_n)$. Then, under the same conditions of Theorem \ref{t}, $E_n$ $\Gamma$-converges to $\mathcal E$ with respect to the topology in \eqref{cTop:R2}, i.e.,
\begin{subequations}
\label{cf:Gconv}
\begin{alignat}2
\label{cf:Glimf}
&\forall \, \kappa \in D(E), \ \forall \, \kappa_n \in D_\circ (E_n), \, \kappa_n \xweakto c \kappa : 
& \ \liminf_{n\to\infty}  E_n (\kappa_n) 
&\geq \mathcal E(\kappa), 
\\
&\forall \, \kappa \in D(E), \ \exists \, \kappa_n \in D_\circ (E_n), \, \kappa_n \xweakto c \kappa : 
& \ \limsup_{n\to\infty}  E_n (\kappa_n) 
&\leq \mathcal E(\kappa).
\label{cf:Glimp}
\end{alignat}
\end{subequations} 
\end{cor}

\begin{proof}
\emph{The liminf inequality \eqref{cf:Glimf}}. 
Given any $\kappa_n \xweakto c \kappa$, let $\bmu_n$ be defined by \eqref{kappan}. Since $\kappa_n$ is tight, for each $s$ the sequence $\mu_{n}^s$ is tight too, and thus $\mu_{n}^s$ converges along a subsequence (not relabelled) to some $\mu^s$ with $\mu^s (\R^2) = \lim_{n \to \infty} n_s/n$. Hence, the obtained measure $\bmu$ satisfies $\bmu \in D(E)$, and thus we find that $\kappa = \sum_{s=1}^S \xi_s \mu^s \in D(\mathcal E)$. We conclude from Theorem \ref{t} that
\begin{equation*}
  \liminf_{n\to\infty}  E_n (\kappa_n)
  \geq E (\bmu)
  \geq \mathcal E (\kappa).
\end{equation*}
\smallskip

\emph{The limsup inequality \eqref{cf:Glimp}}. 
Given $\kappa$, let $(\bmu^\varepsilon)_{\varepsilon > 0} \subset D(E)$ be a minimising sequence for the minimisation problem in \eqref{Em:short}, i.e., $E(\bmu^\varepsilon) \to \mathcal E(\kappa)$ as $\varepsilon \to 0$. Then, for any $\varepsilon > 0$, Theorem \ref{t} provides a sequence $\bmu_n^\varepsilon \xweakto c \bmu^\varepsilon$ as $n \to \infty$ (and thus also $\kappa_n^\varepsilon \xweakto c \kappa$ as $n \to \infty$) for which $\limsup_{n\to\infty}  E_n (\kappa_n^\varepsilon) \leq E(\bmu^\varepsilon)$. 
From a diagonal argument we obtain a sequence $\e_n \to 0$ as $n \to \infty$ such that $\kappa_n := \kappa_n^{\e_n} \weakto \kappa$ as $n \to \infty$ and $\limsup_{n \to \infty}  E_n (\kappa_n) \leq \mathcal E(\kappa)$. 

It is left to show that $\kappa_n \in D_\circ(E_n)$ and $\kappa_n \xweakto c \kappa$ as $n \to \infty$. From the construction of $\bmu_n^{\e_n}$ in the proof of Theorem \ref{t} in terms of the particle positions $x_i^s$ it follows that $\kappa_n \in D_\circ(E_n)$ and that there exists a constant $C > 0$ independent of $n$ such that
\begin{equation*} 
  \dist (x_i^s, \supp \bmu^{\e_n}) \leq C.
\end{equation*}
Then, since $\supp \bmu^{\e_n} \subset \supp \{0\} \cup \kappa$ for all $n$, we have that $\supp \bmu_n^{\e_n}$ is bounded uniformly in $n$, and thus $\kappa_n \xweakto c \kappa$ as $n \to \infty$.
\end{proof}

\section{Riesz potentials}
\label{s:Riesz}

In this section we show how Theorem \ref{t} can be applied to potentials $V^{st}$ whose singularity is described by a Riesz potential. The Riesz potentials are defined for some parameter $0 < a < 2$  by
\begin{equation*} 
  \cV_a(x) := |x|^{-a}
  \quad \text{for all } x \in \R^2 \setminus \{0\}.
\end{equation*}
To focus the attention to $\cV_a$ and not to $V^{st}$ constructed thereof, we consider the easiest non-trivial case given by $2$-species (i.e., $S = 2$) for which $V^{st} := (-1)^{s+t} \cV_a$. This case corresponds to positively and negatively charged particles with a different interaction potential than the Coulomb potential (the Coulomb potential is given by $V(x) = -\log|x|$; we treat this case briefly in Section \ref{s:Riesz:log}). We also provide two examples for regularised potentials $V_\delta^{st}$ in Sections \ref{s:Riesz:LB} and \ref{s:Riesz:convo}.

\subsection{The Riesz potentials fit to Assumption \ref{a:V}}
\label{s:Riesz:aV}

Let $0 < a < 2$. We show that $V^{st} = (-1)^{s+t} \cV_a$ satisfies Assumption \ref{a:V} with $\mathsf K = 1$. With this aim, we recall that 
\begin{equation} \label{cVa:hat}
  \widehat{\cV_a} (\omega)
  = C_a |\omega|^{-(2-a)}
  = C_a |\omega|^{-1 + \tfrac{a}2} |\omega|^{-1 + \tfrac{a}2}
  = C_a' \mathcal F \left( \cV_{1 + \tfrac a2} * \cV_{1 + \tfrac a2} \right) (\omega)
\end{equation}
in the sense of tempered distributions, where $C_a, C_a' > 0$ are explicit constants. Inspired by this observation, we set
\begin{equation*}
  W^{s} := (-1)^s \cV_{b} \psi,
  \qquad \Vreg^{st} := V^{st} - C_a' \overline W^{s} * W^{t}.
\end{equation*}
where $b := 1 + \frac a2$ and $\psi(x) := 1 \wedge (2 - |x|) \vee 0$ is a continuous cut-off function. Since the dependence of the potential on $s$ and $t$ is solely through the multiplication by $-1$, we often set
\begin{align*}
  W &:= W^2 = -W^{1} \geq 0, \\
  V &:= (-1)^{s+t} V^{st} = \cV_a \geq 0, \\
  \Vreg &:= (-1)^{s+t} \Vreg^{st} = C_a' \big( \cV_b * \cV_b - (\cV_{b} \psi) * (\cV_{b} \psi) \big) \geq 0.
\end{align*}

It is easy to see that most of the conditions in Assumption \ref{a:V} on the potentials $\Vreg^{st}$ and $W^s$ are satisfied. The only condition which may not be obvious is whether $\Vreg$ is continuous at the origin. To prove this, we take $x \neq 0$, and write
\begin{equation*}
  \frac1{C_a'} \Vreg (x)
  = \int_{\R^2} |x-y|^{-b} |y|^{-b} \big[ 1 - \psi(x-y) \psi (y) \big] \, dy.
\end{equation*} 
We observe from $2b = 2 + a$ that the right-hand side is finite for $x = 0$. Setting $\Vreg (0)/C_a'$ as this value, we obtain from the constant parts of $\psi$ that
\begin{multline} \label{Va:Vreg:ct}
  \frac{\Vreg (x) - \Vreg (0)}{C_a'} 
  = \int_{A_1(x)} |y|^{-b} \Big( |x-y|^{-b} \big[ 1 - \psi(x-y) \psi (y) \big] - |y|^{-b} \big[ 1 - \psi (y)^2 \big] \Big) \, dy \\
    + \int_{A_2(x)} |y|^{-b} \Big( |x-y|^{-b} - |y|^{-b} \Big) \, dy,
\end{multline}
where the domain
\begin{equation*}
  A_1(x) := \big( B(0,2) \cup B(x,2) \big) \setminus \big( B(0,1) \cap B(x,1) \big)
\end{equation*}
is bounded in size uniformly in $x$, and 
\begin{equation*}
  A_2(x) := \R^2 \setminus \big( B(0,2) \cup B(x,2) \big).
\end{equation*}
It is then easy to verify that both integrals in \eqref{Va:Vreg:ct} tend to $0$ as $x \to 0$. We conclude that $\Vreg$ is continuous at $0$.

\subsection{Example 1 for $V_\delta$: approximation from below}
\label{s:Riesz:LB}

Here we provide our first example for a regularisation $V_\delta$ of $V$ which satisfies Assumption \ref{a:Vd}. We set $\Vreg^\delta := \Vreg$ and take $W_\delta$ as a continuous approximation of $W$ from below. More specifically, we take $W_\delta$ continuous, radially symmetric and decreasing in the radial direction such that $0 \leq W_\delta (x) \uparrow W(x)$ for all $x \neq 0$. Two examples of such $W_\delta$ are
\begin{equation*}
  W_\delta(x) = (|x| + \delta)^{-b} \psi
  \quad \text{and} \quad
  W_\delta(x) = \left\{ \begin{array}{l}
    W(x) \quad \text{if } |x| > \delta \\
    + \text{ affine extension in radial direction}
  \end{array} \right.
\end{equation*} 
It is readily verified that $V_{\delta} = \Vreg^\delta + \overline W_\delta * W_\delta$ satisfies Assumption \ref{a:Vd} (since $0 \leq V_\delta (x) \leq V(x)$, the dominator $U$ in Assumption \ref{a:Vd}\ref{a:Vd:dom} can simply be taken as $V$). 

\subsection{Example 2 for $V_\delta$: mollifying $V$}
\label{s:Riesz:convo}

Here we provide our second example for $V_\delta$, which is constructed by mollifying $V$. With this aim, let $\varphi \in C (\R^2)$ be radially symmetric with compact support, $\int_{\R^2} \varphi = 1$ and
\begin{equation*}
  \Phi := \overline \varphi * \varphi \geq 0.
\end{equation*} 
For $\delta>0$ we define $\varphi_\delta(x) := \varphi (x / \delta)/\delta^2$ and $\Phi_\delta(x) := \Phi (x / \delta)/\delta^2$, and note that $\Phi_\delta = \overline \varphi_\delta * \varphi_\delta$. Finally, we set
\begin{equation*}
  W_\delta := \varphi_\delta * W, 
  \quad
  \Vreg^\delta := \Phi_\delta * \Vreg
  \quad \text{and} \quad
  V_\delta
  := \Vreg^\delta + \overline W_\delta * W_\delta.
\end{equation*}

Assumption \ref{a:Vd}\ref{a:Vd:W-form} is satisfied by construction. By observing that
\begin{equation} \label{Vphiphi}
  V_\delta
  = \Phi_\delta * \Vreg
    + (\overline \varphi_\delta * \overline W)
    * (\varphi_\delta * W) \\
  = (\overline \varphi_\delta * \varphi_\delta) * \big( \Vreg + \overline W * W \big)
  = \Phi_\delta * V,
\end{equation}
it is clear that Assumptions \ref{a:Vd}\ref{a:Vd:reg:even} and \ref{a:Vd}\ref{a:Vd:uf:conv} are satisfied too. To show Assumption \ref{a:Vd}\ref{a:Vd:dom}, we construct a dominator $U$. First, we take any $x \in \R^2$ with $|x| = 1$, and note by the radial symmetry of $V_\delta$ that the function
\begin{equation*}
  \gamma : (0, \infty) \to \R,
  \qquad \gamma(\delta) := V_{\delta}(x)
\end{equation*}
is independent of the choice of $x$. Moreover, $\gamma$ is continuous on $(0,\infty)$ with $\gamma \to 1$ as $\delta \to 0$ and $\gamma \to 0$ as $\delta \to \infty$. Thus,
\begin{equation*}
  M := \sup_{0 < \delta < \infty} V_{\delta}(x) \geq 1
\end{equation*}
is finite.

To extend to $\alpha x \in B(0,1)$ for any $0 < \alpha < 1$, we use \eqref{Vphiphi} to compute
\begin{multline*}
  V_{\delta} (\alpha x)
  = \int_{\R^2} |\alpha x - y|^{-a} \Phi_\delta (y) \, dy 
  = \int_{\R^2} |\alpha x - \alpha z|^{-a} \Phi_\delta (\alpha z) \, \alpha^2 dz \\
  = |\alpha|^{-a} \int_{\R^2} |x - z|^{-a} \Phi_{\delta / \alpha} (z) \, dz 
  = \cV(\alpha x) V_{\delta / \alpha} (x).
\end{multline*}
Hence, for any $x \in B(0,1) \setminus \{0\}$, it follows from the argument above that
\begin{equation*}
  \sup_{0 < \delta < \infty} V_{\delta} (x)
  = \cV_a(x) \sup_{0 < \delta < \infty} V_{ \delta/|x| } \Big( \frac x{|x|} \Big)
  = \cV_a(x) \sup_{0 < \varepsilon < \infty} V_{\varepsilon} \Big( \frac x{|x|} \Big)
  = M \cV_a(x) =: U(x).
\end{equation*}

\subsection{The logarithmic potential}
\label{s:Riesz:log}

For $V(x) = -\log|x|$, the energy $E_n$ models both a collection of screw dislocations as well as a collection of vortices, and is therefore of independent interest. Since the interaction potential for edge dislocations in \eqref{V:bibj} is a more complicated potential of logarithmic type, we refer to Section \ref{s:V:edge} for fitting $V$ to Assumption \ref{a:V} and for constructing regularisations $V_\delta$ of it which satisfy Assumption \ref{a:Vd}. In particular, for the decomposition of $V$ in \eqref{V:W-form}, the analysis in Section \ref{s:V:edge:aV} can be somewhat simplified by using the structure on screw dislocations obtained in \cite{BlassMorandotti17}.

\section{The case of edge dislocations} 
\label{s:V:edge}

In this section we show that Theorem \ref{t} applies to the model case of edge dislocations. Similar to Section \ref{s:Riesz}, we first show in Section \ref{s:V:edge:aV} that $V(\cdot \, ;b_i, b_j)$ as defined in \eqref{V:bibj} satisfies Assumption \ref{a:V}. Then, in Sections \ref{s:V:edge:perf} and \ref{s:V:edge:convo} we give two examples of regularisations $V_\delta$ of $V$ which satisfy Assumption \ref{a:Vd}.

\subsection{The interaction potential between edge dislocations}
\label{s:V:edge:aV}

For the analysis in this section, it is more convenient to express $V(\cdot \, ;b_i, b_j)$ in terms of the angles $\phi_i$ and $\phi_j$ which the respective Burgers vectors $b_i$ and $b_j$ make with the horizontal axis, i.e., $\cos \phi_i = b_i \cdot e_1$ and $\cos \phi_j = b_j\cdot e_1$. Then, we rewrite \eqref{V:bibj} as
\begin{equation} \label{V:phipsi}
    V(x; \phi_i,\phi_j)  
    = - \cos (\phi_i - \phi_j) \log | x | + \cos^2 \Big( \arccos \Big[ \frac x{|x|} \cdot e_1 \Big] - \frac{\phi_i + \phi_j}2 \Big).
\end{equation}

To find a decomposition of the form \eqref{V:W-form}, we use the interpretation of $V$ as a renormalisation of the elastic energy in the surrounding elastic medium. This derivation is done in \cite{CermelliLeoni06}. Here, we briefly recall the setting and several results from \cite{CermelliLeoni06}. Let the elastic domain be given by a simply connected, smooth and {bounded} domain $\Omega \subset \R^2$. The interaction potential between two dislocations at $x, y \in \Omega$ with Burgers vectors given in terms of the angles $\phi$ and $\psi$ is given by
\begin{equation} \label{VOmega}
V_\Omega (x,y; \phi, \psi) 
= \int_{\Omega} \C{} K^\phi (z - x) : K^\psi (z - y) \, dz.
\end{equation}
In \eqref{VOmega}, $\mathbb{C}$ is the tensor of isotropic linearised elasticity defined, for $F\in \R^{2\times 2}$, as
$$
  \mathbb{C}F := \lambda (\tr F) \Id + 2 \mu \sym F. 
$$ 
The product
$$
  A:B = \sum_{i,j=1}^2 A_{ij} B_{ij} = \tr (A^T B) \quad \text{for } A, B\in \R^{2\times 2}
$$
denotes the Frobenius inner product between matrices. Finally, the function \\
$K^\phi \in L_{\text{loc}}^1 (\R^2; \R^{2 \times 2})$ is a solution of 
\begin{equation} \label{Kphi:eqn}
\begin{cases}
  \Curl K^\phi = b_\phi \delta_0 
  & \text{in } \R^2,
  \\
  \Div \C{} K^\phi = 0
  & \text{in } \R^2,
\end{cases}
\end{equation}
in distributional sense, where $b_\phi$ is the Burgers vector related to $\phi$. In other words, the function $K^\phi$ describes the strain in $\R^2$ induced by a dislocation at $0$ with Burgers vector $b_\phi$. The explicit expression for $K^\phi$ is given in \eqref{for:Kphi:explicit}. Here, it is enough to know that
\begin{equation} \label{Kphi:props:basic}
  K^\phi (x) = \frac1{|x|} \tilde K^\phi\Big( \frac x{|x|} \Big),
  \qquad K^\phi (-x) = K^\phi (x)
\end{equation}
for some regular, $\R^{2 \times 2}$-valued function $\tilde K^\phi$.

We note that the expression of $V_\Omega$ in \eqref{VOmega} resembles the convolution terms in \eqref{V:W-form}. Indeed, for a convenient choice of $\Omega$, we show in Remark \ref{remark:Wk} that we can rewrite it as a sum of convolutions of some functions $W_k^\phi$. Moreover, In \cite[Prop.~5.2]{CermelliLeoni06} it is already shown that
\begin{equation}\label{CKK:O1}
  V_\Omega (x,y; \phi, \psi) 
  = -\frac{ \mu ( \lambda + \mu ) }{ \pi ( \lambda + 2 \mu ) } \cos(\phi - \psi) \log |x - y| + \mathcal O (1)
\end{equation}
as $|x - y| \to 0$, as long as $x, y \in \Omega$ stay away from $\partial \Omega$. Note that this expression is consistent, up to the multiplicative constant $\frac{ \mu ( \lambda + \mu ) }{ \pi ( \lambda + 2 \mu ) }$, with that of $V$ in \eqref{V:phipsi}.

Next we build further on these observations to find a decomposition of $V$ as in \eqref{V:W-form}. In particular, in Proposition \ref{p:V:CKK} we update the expansion in \eqref{CKK:O1} up to $o(1)$ and show that $V_\Omega$ captures the singularity of $V$ at $0$. This is non-trivial, because the second term in \eqref{V:phipsi} is discontinuous at $0$. 

\begin{prop}[Decomposition of $V$] \label{p:V:CKK}
Let $V$ be as in \eqref{V:phipsi}. Then, for all $\phi, \psi \in [0, 2\pi)$ there exists $\Vreg ( \, \cdot \, ; \phi, \psi ) \in C (\R^2)$ such that for all $x,y \in \R^2$ with $x \neq y$,
\begin{equation} \label{V:CKK} 
V(x - y; \phi, \psi) =  \frac{ \pi ( \lambda + 2 \mu ) }{ \mu ( \lambda + \mu ) } \int_{B(x,1) \cap B(y,1)} \C K^\phi (z - x) : K^\psi (z - y) \, dz 
+ \Vreg (x - y; \phi, \psi),
\end{equation}
where $\C$ and $K^\phi$ are as in \eqref{VOmega}, and $B(x,1)$ and $B(y,1)$ are the balls of radius $1$ centred at $x$ and $y$ respectively. Moreover, for a  bounded Lipschitz domain $\Omega\subset \R^2$, it holds that
\begin{equation} \label{V:VOm:gOm}
 V_\Omega (x,y; \phi, \psi) 
 = \frac{ \mu ( \lambda + \mu ) }{ \pi ( \lambda + 2 \mu ) } V(x - y; \phi, \psi) + g_\Omega (x, y; \phi, \psi),
\end{equation}
for some $g_\Omega ( \, \cdot \, , \cdot \, ; \phi, \psi ) \in C (\Omega^2)$.
\end{prop}

The decomposition in \eqref{V:VOm:gOm} is essentially the splitting of the Green's function relative to $\Omega$, in the linearised elasticity context, into its `fundamental solution' part, and a smooth term carrying the information on the boundary contribution. The proof of Proposition \ref{p:V:CKK} is left to Appendix \ref{a:Prop:pf}. 

\begin{rem}[Higher regularity for $\Vreg$] 
The choice of the integration domain in \eqref{V:CKK} is equivalent to multiplying $K^\phi$ and $K^\psi$ with the cut-off function $\indicatornoacc{ B(0,1) }$. Alternatively, one could use more regular cut-off functions, which would result in a higher regularity for $\Vreg$ outside the origin.
\end{rem}

\begin{rem}[Construction of $W_k^\phi$] \label{remark:Wk}
By Proposition \ref{p:V:CKK} , since $\mathbb C$ is symmetric and positive definite, there exists a symmetric linear operator $\mathbb D$ such that $\mathbb D^2 = \mathbb C$. We use this to rewrite the integral in \eqref{V:CKK} as
\begin{align*} 
&\int_{B (x, 1) \cap B (y, 1)} \C{} K^\phi (z - x) : K^\psi (z - y) \, dz \\ 
&= \int_{\R^2} \lrhaa{ \mathbb D K^\phi \indicatornoacc{ B (0, 1) } } (z-x) : \lrhaa{ \mathbb D K^\psi \indicatornoacc{ B (0, 1) } } (z-y) \, dz \\ 
&= \sum_{i,j=1}^2 \Bighaa{ \bighaa{ (\mathbb D \overline K^\phi)_{ij} \indicatornoacc{ B (0, 1) } } * \bighaa{ (\mathbb D K^\psi)_{ij} \indicatornoacc{ B (0, 1) } } } (x-y) \\ 
&=: \frac{ \mu ( \lambda + \mu ) }{ \pi ( \lambda + 2 \mu ) } \sum_{k=1}^4 \bighaa{ \overline W_k^\phi * W_k^\psi } (x-y).
\end{align*}
\end{rem}

In conclusion, Proposition \ref{p:V:CKK} and Remark \ref{remark:Wk} provide a decomposition of $V$ of the type \eqref{V:phipsi}. It is readily checked that this decomposition satisfies Assumption \ref{a:V}. 

\subsection{Example 1 for $V_\delta$: perforating the domain}
\label{s:V:edge:perf}

Given the decomposition of $V$ as in Section \ref{s:V:edge:aV}, we construct a regularisation $V_\delta$ of it which satisfies Assumption \ref{a:Vd}. Inspired by the idea to perforate the domain as in \eqref{Ed:GLP10}, we set 
\begin{equation*}
  W_{\delta, k}^\phi := W_k^\phi \indicatornoacc{B(0,\delta)^c} \in L^\infty (\R^2),
  \qquad \Vreg^{\delta, st} := \Vreg^{st}.
\end{equation*}  
We recall from \eqref{Kphi:props:basic} that $W_k^\phi \in L^1 (\R^2)$, that $\supp W_k^\phi = \overline{ B(0,1) }$, that $W_k^\phi$ is odd and that $|W_k^\phi(x)| \leq C/|x|$. 
  
Next we show that $V_\delta(x; \phi, \psi)$ as constructed from $W_{\delta, k}^\phi$ and $\Vreg^{\delta}$ through \eqref{V:W-form-d} satisfies Assumption \ref{a:Vd}. Conditions \ref{a:Vd:uf:conv} and \ref{a:Vd:W-form} are readily verified.
Condition \ref{a:Vd:reg:even} follows from the oddness of $W_{\delta, k}^\phi$ and from the observation that $W_{\delta,k}^\phi \in L^1 (\R^2) \cap L^\infty (\R^2)$. 
To show Condition \ref{a:Vd:dom}, we fix some $0 < \e < 1$, and use \eqref{Kphi:props:basic} and \eqref{cVa:hat} to estimate for all $x \in \R^2 \setminus \{0\}$
\begin{multline*}
  \big| W_{\delta,k}^\phi * W_{\delta,k}^\psi \big|(x)
  \leq \big( | W_{k}^\phi | * | W_{k}^\psi | \big)(x)
  \leq C^2 \bigg( \frac{ \indicatornoacc{B(0,1)} }{|\cdot|} * \frac{ \indicatornoacc{B(0,1)} }{|\cdot|} \bigg)(x) \\
  \leq C^2 \bigg( \frac{ \indicatornoacc{B(0,1)} }{|\cdot|^{1+\e}} * \frac{ \indicatornoacc{B(0,1)} }{|\cdot|^{1+\e}} \bigg)(x)
  \leq \tilde C |x|^{-2\e}.
\end{multline*}
\medskip

While the regularisation above fits, the equivalent of $E_\delta$ in \eqref{Ed:GLP10} would not fit. The reason is that in $E_\delta$, the medium is perforated around each dislocation. Then, the interaction potential between two dislocations also depends on the position of all others, and is therefore no more a pairwise interaction potential. Under a further restriction in which the dislocations are well separated, this perturbation to the interaction potential is expected to be negligible.

\subsection{Example 2 for $V_\delta$: mollifying $V$}
\label{s:V:edge:convo}

As in Section \ref{s:Riesz:convo}, we set
\begin{equation*}
  W_{\delta,k}^\phi = \varphi_\delta * W_{k}^\phi 
  \quad \text{and} \quad
  \Vreg^\delta(\cdot \, ; \phi, \psi) = \Phi_\delta * \Vreg (\cdot \, ; \phi, \psi),
\end{equation*}
where we assume that the mollifiers $\varphi_\delta$ and $\Phi_\delta$ satisfy the same conditions as in Section \ref{s:Riesz:convo}. Again, we set $V_\delta(x; \phi, \psi)$ as in \eqref{V:W-form-d}, observe as in Section \ref{s:Riesz:convo} that 
\begin{equation*}
  V_\delta = \Phi_\delta * V,
\end{equation*}
and notice that Assumptions \ref{a:Vd}\ref{a:Vd:reg:even}--\ref{a:Vd:W-form} are obviously satisfied.

Assumption \ref{a:Vd}\ref{a:Vd:dom} is less trivial to check. We prove that it is satisfied by constructing a dominator $U$. By the uniform convergence in Assumption \ref{a:Vd}\ref{a:Vd:uf:conv} it is sufficient to construct $U (x)$ for any $|x|$ and $\delta$ small enough. With this aim, let $x \in B(0,\frac12)$ and let $\delta > 0$ be such that $\supp \Phi_\delta \subset B(0, \tfrac12)$. Observing from \eqref{V:phipsi} that $|V(x; \phi, \psi)| \leq -\log | x | + 1$, we obtain from $\Phi_\delta \geq 0$ that
\begin{equation*}
   \big| V_\delta \big| 
   \leq |V| * \Phi_\delta 
   \leq - \log |\cdot| * \Phi_\delta + \int_{\R^2} \Phi_\delta
   = - \log |\cdot| * \Phi_\delta + 1
   \quad \text{on } B \Big( 0, \frac12 \Big).
\end{equation*}
To bound the convolution term, we first recall that the mean-value property of the subharmonic function $\log |\cdot|$ implies 
$$
  \log|x| \leq \frac1{2\pi r} \int_{\partial B(x,r)} \log|y| \, dy
  \quad \text{for all } r > 0.
$$
Then, since $\Phi_\delta$ is radially symmetry with support in $B(0,\frac12)$, 
\begin{align*}
  \big( \log | \cdot | * \Phi_\delta \big) (x) 
  &= \int_0^{\tfrac12} \bigg( \int_0^{2 \pi} \log \bigg| x - r \begin{bmatrix}
  \cos \theta \\ \sin \theta
  \end{bmatrix} \bigg| \, d \theta \bigg) \Phi_\delta (r) r \, dr \\
  &\geq 2 \pi \log | x | \int_0^{\tfrac12} \Phi_\delta (r) \, dr
  = \log | x |.
\end{align*}
Hence, $U(x) := -\log | x | + 1$ is a dominator for $V_\delta(\cdot \, ; \phi, \psi)$ for any $\phi, \psi \in [0, 2\pi)$.
\medskip

\begin{rem}[Mollifying the dislocation core]
Note that the choice of regularisation of $V$ by mollification is equivalent, at the level of the energy, to mollifying the dislocation cores by $\varphi_{\delta}$, and plugging the resulting density $\varphi_{\delta_n} * \bmu_n$ into the energy $E$ obtained as the $\Gamma$-limit. Indeed, from \eqref{En:mun:intro} and \eqref{E:intro} we find that
\begin{align*}
  E_n(\bmu_n)&= \sum_{s,t=1}^S \int_{\R^2} \big( V_{\delta_n}( \cdot \, ; \xi_s, \xi_t) * \mu_n^t \big) \, d \mu_n^s \\
  &= \sum_{s,t=1}^S \int_{\R^2} \big( V ( \cdot \, ; \xi_s, \xi_t) * (\varphi_{\delta_n} * \mu_n^t ) \big) (x) \, (\varphi_{\delta_n} * \mu_n^s ) (x) \, dx\\
  &= E ( \varphi_{\delta_n} * \bmu_n ).
\end{align*}

\end{rem}


\appendix

\section{Proof of Proposition \ref{p:V:CKK}}
\label{a:Prop:pf}

The proof of Proposition \ref{p:V:CKK} is a quantitative version of the proof of \cite[Prop.~5.2]{CermelliLeoni06}. Since it is computationally heavy, we first introduce some convenient notation and recall some formulas from \cite{CermelliLeoni06}.

For an angle $\phi \in [0, 2\pi)$, we denote with $J_\phi\in SO(2)$ the matrix corresponding to a counter-clockwise rotation by $\phi$. We further set $\hat r_\phi := J_\phi e_1 \in \mathbb S$. We take $K^\phi$ and $\C$ as in \eqref{VOmega}, where
\begin{multline} \label{for:Kphi:explicit}
  K^\phi ( r, \theta )
  = \frac{ 1 }{ 2 \pi (\lambda + 2 \mu) } \frac 1 r \Big[ 
     \mu \sin ( \phi - \theta ) \, \hat r_\theta \otimes \hat r_\theta 
     + (2 \lambda + 3 \mu) \cos ( \phi - \theta ) \, \hat r_\theta \otimes \hat r_{\theta + \tfrac\pi2} \\
     - \mu \cos ( \phi - \theta ) \, \hat r_{\theta + \tfrac\pi2} \otimes \hat r_\theta
     + \mu \sin ( \phi - \theta ) \, \hat r_{\theta + \tfrac\pi2} \otimes \hat r_{\theta + \tfrac\pi2} \Big]
\end{multline}
in polar coordinates. In \cite[Rem.~3.2]{CermelliLeoni06} it is shown that this expression for $K^\phi$ is a solution to \eqref{Kphi:eqn}. We further set $K := K^0$ and $K_x^\phi (z) = K^\phi(z - x)$. Since $\C$ is rotationally invariant, we have for any angle $\phi$ and any $A, B \in \R^{2 \times 2}$ that
\begin{equation} \label{fa:C:Rinv}
 \C (J_\phi A J_{-\phi}) : (J_\phi B J_{-\phi}) = \C A : B. 
\end{equation} 

The stress corresponding to $K = K^0$ is
\begin{multline*}
  \C K ( r, \theta )
  = \frac{ \mu (\lambda + \mu) }{ \pi (\lambda + 2 \mu) } \frac 1 r \Big[ 
     - (\sin \theta ) \, \hat r_\theta \otimes \hat r_\theta 
     + (\cos \theta ) \, \big( \hat r_\theta \otimes \hat r_{\theta + \tfrac\pi2}
     + \hat r_{\theta + \tfrac\pi2} \otimes \hat r_\theta \big) \\
     - (\sin \theta ) \, \hat r_{\theta + \tfrac\pi2} \otimes \hat r_{\theta + \tfrac\pi2} \Big].
\end{multline*}
Outside of the branch cut parametrised by $\alpha e_1$ for $\alpha \geq 0$, $K^\phi$ is the gradient of
\begin{equation*}
  w^\phi ( r, \theta ) 
  = \frac1{2\pi} \Big[ \theta \hat r_\phi 
     + \frac{ \mu }{ \lambda + 2 \mu } (- \log r) \, \hat r_{\phi + \tfrac\pi2}
     - \frac{ \lambda + \mu }{ 2 (\lambda + 2 \mu) } \big( \sin (\phi - \theta) \hat r_\theta + \cos (\phi - \theta) \hat r_{\theta + \tfrac\pi2} \big) \Big].
\end{equation*}
Along the branch cut, we observe that the jump is given by
$$\llbracket w^\phi \rrbracket (r) 
:= w^\phi (r, 0+) - w^\phi (r, 2 \pi-)
= - \hat r_\phi.$$

Next we prove the auxiliary identity
\begin{equation} \label{for:Kphi:ito:rotations}
  K^\phi = J_\phi ( K \circ J_{-\phi} ) J_{-\phi}
  \quad \text{for all } \phi \in [0, 2\pi).
\end{equation}
From \eqref{for:Kphi:explicit} we observe that $K^\phi (r, \theta)$ consists of a sum of four terms, each of the form $\frac1r f (\theta- \phi) \hat r_{\theta + \alpha} \otimes \hat r_{\theta + \beta}$, with $f$ a scalar function and $\alpha, \beta \in \{0, \frac\pi2 \}$. From the identity
$$
  J_\phi ( \hat r_{\theta} \otimes \hat r_{\vartheta} ) J_{-\phi} 
  = (J_\phi \hat r_{\theta}) \otimes (J_{-\phi}^T \hat r_{\vartheta} ) 
  = \hat r_{\theta + \phi} \otimes \hat r_{\vartheta + \phi}
  \quad \text{for all } \phi, \theta, \vartheta,
$$
we observe that    
\begin{equation*}
  J_\phi \big( (f \, \hat r_{\cdot + \alpha} \otimes \hat r_{\cdot + \beta} ) \circ J_{-\phi} \big) (\theta) J_{-\phi}
  = J_\phi \Big( f (\theta - \phi) ( \hat r_{\theta + \alpha - \phi} \otimes \hat r_{\theta + \beta - \phi} ) \Big) J_{-\phi}
  = f (\theta - \phi) \, \hat r_{\theta + \alpha} \otimes \hat r_{\theta + \beta}
\end{equation*}
and the identity \eqref{for:Kphi:ito:rotations} follows.

\begin{proof}[Proof of Proposition \ref{p:V:CKK}]
First we prove \eqref{V:VOm:gOm} for $\Omega = B(x,1)$. We start by simplifying  the expression in \eqref{VOmega} by using \eqref{fa:C:Rinv} and \eqref{for:Kphi:ito:rotations}. This yields
\begin{align*}
  &\int_{B(x,1)} \C K^\phi (z - x) : K^\psi (z - y) \, dz \\
  &= \int_{B(0,1)} \C K^\phi (\zeta) : K^\psi (\zeta - (y-x)) \, d\zeta \\
  &= \int_{B(0,1)} \C K (J_{-\phi} \zeta) : \Big( J_{\psi-\phi} K \big( J_{-\psi}[\zeta - (y-x)] \big) J_{\phi-\psi} \Big) \, d\zeta \\
  &= \int_{B(0,1)} \C K (J_{-\phi} \zeta) : K^{\psi - \phi} \big( J_{-\phi}[\zeta - (y-x)] \big) \, d\zeta \\
  &= \int_{B(0,1)} \C K (\eta) : K^{\psi - \phi} \big( \eta -  J_{-\phi} (y-x) \big) \, d\eta.
\end{align*}
By further setting $\tilde \phi := \psi - \phi$ and $\tilde x := J_{\phi} (x-y)$, we obtain by changing coordinates that
\begin{equation} \label{prop:connn:V:CKK:pf:05}
  \int_{B(x,1)} \C K^\phi (z - x) : K^\psi (z - y) \, dz
  = \int_{B(0,1)} \C K (z) : K^{\tilde \phi} \big( z -  \tilde x \big) \, dz
  = \int_{B(0,1)} \C K : K_{\tilde x}^{\tilde \phi}.
\end{equation}
In the remainder we focus only on the right-hand side. We remove the tildes for convenience.

In terms of dislocation interactions, \eqref{prop:connn:V:CKK:pf:05} corresponds to the setting in Figure \ref{fig:prop:conn:CKK:V}, where $x = (r, \theta)$ are the polar coordinates. To apply integration by parts on $\int_{B(0,1)} \C K : K_x^\phi$, we use that $K_x^\phi = \nabla w_x^\phi$ away from a branch cut connecting $x$ to $\partial B(0,1)$. We take the particular branch cut $\Gamma$ as in Figure \ref{fig:prop:conn:CKK:V}, which is parametrised by $\gamma(s) = s \hat r_\theta$ with $ s \in [r, 1]$. Then, using $\Div \C K = 0$, we obtain
\begin{equation} \label{prop:connn:V:CKK:pf:1}
  \int_{B(0,1)} \C K : K_x^\phi 
  = \int_{B(0,1) \setminus \Gamma} \C K : \nabla w_x^\phi
  = \int_{\partial B(0,1)} w_x^\phi \cdot \C K \cdot n 
    + \int_\Gamma \llbracket w_x^\phi \rrbracket \cdot \C K \cdot n.
\end{equation}

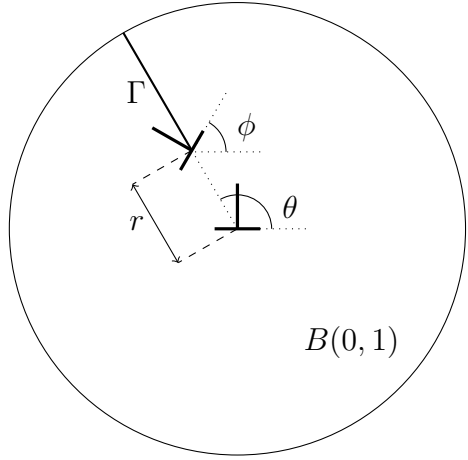
\begin{figure}[h!]
\centering
\begin{tikzpicture}[scale=3]
    \def \dlc {0.1}
    \def \sqrtthree {1.70}
    \def \r {0.4}
    \def \x {\r * 0.5}
    \def \y {\r * \sqrtthree * 0.5}
    
    \draw (0,0) circle (1);
    \draw[dotted] (0.3, 0) -- (0, 0);
    \draw[very thick] (- \dlc, 0) -- (\dlc, 0);
    \draw[very thick] (0, 0) -- (0, 2*\dlc);
    \draw (0.15,0) node[anchor = south west] {$\theta$} arc (0:120:0.15); 
    \draw (0.5, -0.5) node{$B(0,1)$};
    
    \begin{scope}[shift={(-\x, \y)}, scale=1]
      \draw[dotted] (0.3, 0) -- (0, 0);
      \begin{scope}[rotate = 60, scale=1]
        \draw[dotted] (0.3, 0) -- (0, 0); 
      \end{scope}
      \draw (0.15,0) node[anchor = south west] {$\phi$} arc (0:60:0.15);     
    \end{scope}
    
	\begin{scope}[rotate = 120, scale=1]
	  \draw[thick] (\r,0) -- (1,0) node[midway, left] {$\Gamma$}; 
	  \draw[dotted] (0,0) -- (\r,0);  
	  \draw[dashed] (0,0) -- (0, 0.3); 
	  \draw[dashed] (\r,0) -- (\r, 0.3);   
	  \draw[<->] (0, 0.3) -- (\r, 0.3) node[midway, left] {$r$};
      \begin{scope}[shift={(\r, 0)}, rotate = 300, scale=1]
        \draw[very thick] (- \dlc, 0) -- (\dlc, 0);
        \draw[very thick] (0, 0) -- (0, 2*\dlc); 
      \end{scope}
    \end{scope}
\end{tikzpicture} \hspace{10mm}
\caption{Schematic setting of \eqref{prop:connn:V:CKK:pf:1}.}
\label{fig:prop:conn:CKK:V}
\end{figure}

The second integral in \eqref{prop:connn:V:CKK:pf:1} equals
\begin{multline*}
  \int_\Gamma \llbracket w_x^\phi \rrbracket \cdot \C K \cdot n
  = \int_r^1 (-\hat r_\phi) \C K (\rho, \theta) \cdot \hat r_{\theta - \tfrac\pi2} \, d\rho \\
  = \frac{ \mu (\lambda + \mu) }{ \pi (\lambda + 2 \mu) } 
    \big( \cos \theta \cos (\phi - \theta) - \sin \theta \cos (\phi - \theta) \big)
    \int_r^1 \frac 1 \rho \, d\rho 
  = \frac{ \mu (\lambda + \mu) }{ \pi (\lambda + 2 \mu) } ( \cos \phi ) (- \log r).
\end{multline*}
We split the first integral in \eqref{prop:connn:V:CKK:pf:1} as
\begin{multline} \label{prop:connn:V:CKK:pf:2}
  \int_{\partial B(0,1)} w_x^\phi \cdot \C K \cdot n 
  = \int_{\partial B(0,1)} w^\phi \cdot \C K \cdot n
    + \int_{\partial B(0,1) \setminus B( \hat r_\theta,2r) } (w_x^\phi - w^\phi) \cdot \C K \cdot n \\
    + \int_{\partial B(0,1) \cap B( \hat r_\theta,2r ) } (w_x^\phi - w^\phi) \cdot \C K \cdot n.
\end{multline}
Since $w_x^\phi$ converges uniformly to $w^\phi$ as $x \to 0$ in any compact set which is disjoint with the branch cut, the second integral in the right-hand side of \eqref{prop:connn:V:CKK:pf:2} is $\mathcal O (r)$ uniformly in $\phi$. The third integral is smaller than $C r$ because $w_x^\phi$, $w^\phi$ and $K$ are bounded in any annulus where the inner disk contain $0$ and $x$. 

We rewrite the first integral in \eqref{prop:connn:V:CKK:pf:2} as
\begin{equation*}
  \int_{\partial B(0,1)} w^\phi \cdot \C K \cdot n 
  = \int_\theta^{\theta + 2 \pi} w^\phi (1, \vartheta) \cdot \C K (1, \vartheta) \cdot \hat r_\vartheta \, d \vartheta.
\end{equation*}
To evaluate the integrand, we first compute
\begin{equation*}
  \C K (1, \vartheta) \cdot \hat r_\vartheta 
  = \frac{ \mu (\lambda + \mu) }{ \pi (\lambda + 2 \mu) } \Big[ 
     - (\sin \vartheta ) \, \hat r_\vartheta 
     + (\cos \vartheta ) \, \hat r_{\vartheta + \tfrac\pi2} \Big].
\end{equation*}
Then, we rewrite the integrand as
\begin{align*}
  w^\phi (1, \vartheta) \cdot \C K (1, \vartheta) \cdot \hat r_\vartheta 
  &= \frac{ \mu (\lambda + \mu) }{ 2 \pi^2 (\lambda + 2 \mu) } \bigg[ 
    \vartheta \big( - \sin \vartheta \cos (\phi - \vartheta) + \cos \vartheta \sin (\phi - \vartheta) \big) \\
  &\qquad \qquad \qquad - \frac{ (\lambda + \mu) }{ 2 (\lambda + 2 \mu) } \big( - \sin \vartheta \sin (\phi - \vartheta) + \cos \vartheta \cos (\phi - \vartheta) \big) \bigg] \\
  &= \frac{ \mu (\lambda + \mu) }{ 2 \pi^2 (\lambda + 2 \mu) } \bigg[ 
    \vartheta \sin (\phi - 2 \vartheta) - \frac{ (\lambda + \mu) }{ 2 (\lambda + 2 \mu) } \cos \phi \bigg],
\end{align*}
from which we easily see that
\begin{equation*}
  \int_\theta^{\theta + 2 \pi} w^\phi (1, \vartheta) \cdot \C K (1, \vartheta) \cdot \hat r_\vartheta \, d \vartheta
  = \frac{ \mu (\lambda + \mu) }{ 2 \pi (\lambda + 2 \mu) } \bigg[ 
    \cos (\phi - 2 \theta) - \frac{ (\lambda + \mu) }{ (\lambda + 2 \mu) } \cos \phi \bigg].
\end{equation*}

Collecting our results, we obtain from \eqref{prop:connn:V:CKK:pf:1} that
\begin{multline*} 
  \frac{ \pi ( \lambda + 2 \mu ) }{ \mu (\lambda + \mu) } \int_{B(0,1)} \C K : K_x^\phi 
  = \frac12 \cos (2 \theta - \phi) - \frac{ (\lambda + \mu) }{ 2(\lambda + 2 \mu) } \cos \phi - (\cos \phi) \log r + \tilde g_\phi (x) \\
  = V(x; \phi, 0) + g_\phi (x)
\end{multline*}
for some $\tilde g_\phi, g_\phi \in C(\R^2)$ satisfying $g_\phi (x) \to 0$ as $x \to 0$ uniformly in $\phi$. Thus, translating back to the original coordinates via \eqref{prop:connn:V:CKK:pf:05}, we obtain
\begin{multline*}
  \frac{ \pi ( \lambda + 2 \mu ) }{ \mu (\lambda + \mu) } \int_{B(x,1)} \C K^\phi (z - x) : K^\psi (z - y) \, dz
  = \frac{ \pi ( \lambda + 2 \mu ) }{ \mu (\lambda + \mu) } \int_{B(0,1)} \C K (z) : K^{\tilde \phi} \big( z -  \tilde x \big) \, dz \\
  = V(\tilde x; \tilde \phi, 0) + g_{\tilde \phi} (\tilde x)
  = V(x - y; \phi, \psi) + g_{\phi - \psi} \big(J_{\phi} (x-y) \big).
\end{multline*}
We conclude that \eqref{V:VOm:gOm} holds for $\Omega = B(0,1)$.

Next we prove \eqref{V:VOm:gOm} for any bounded Lipschitz domain $\Omega$. By the translation invariance, it is not restrictive to assume $0 \in \Omega$. We write
\begin{equation*}
  \int_\Omega \C K : K_x^\phi
  = \int_{B(0,1)} \C K : K_x^\phi
    + \int_{\Omega \setminus B(0,1)} \C K : K_x^\phi
    - \int_{B(0,1) \setminus \Omega} \C K : K_x^\phi
\end{equation*}
and show that the second and third integrals in the right-hand side are continuous in $x$ at $0$. We focus on the second integral, since the argument for the third works analogously. We rewrite it as
\begin{equation} \label{prop:connn:V:CKK:pf:3}
  \int_{\Omega \setminus B(0,1)} \C K : K_x^\phi
  = \int_{\Omega \setminus B(0,1)} \C K : K^\phi
    + \int_{\Omega \setminus B(0,1)} \C K : (K_x^\phi - K^\phi).
\end{equation}
Since $\Omega \setminus B(0,1)$ is contained in some annulus centred at $0$, we find that $K$ and $K^\phi$ are uniformly bounded on $\Omega \setminus B(0,1)$, and thus the first integral in the right-hand side of \eqref{prop:connn:V:CKK:pf:3} is finite and an independent of $x$. Moreover, $K_x^\phi \to K^\phi$ uniformly in any annulus centred at $0$ as $x \to 0$. Therefore, the value of the second integral is continuous in $x$ at $0$. This concludes the proof of \eqref{V:VOm:gOm}.

We establish \eqref{V:CKK} with similar arguments. Again, we change coordinates to rewrite the first term in the right-hand side of \eqref{V:CKK} as 
\begin{equation*}
  \int_{B(0,1) \cap B(x,1)} \C K : K_x^\phi
  = \int_{B(0,1)} \C K : K_x^\phi
    - \int_{B(0,1) \setminus B(x,1)} \C K : K_x^\phi.
\end{equation*}
For the first term we apply \eqref{V:VOm:gOm}. The second term vanishes as $x \to 0$, because $| B(0,1) \setminus B(x,1) | \to 0$ as $x \to 0$ and $K$ and $K_x^\phi$ are uniformly bounded on $B(0,1) \setminus B(x,1)$ for all $|x|$ small enough.
\end{proof}


\noindent\textbf{Acknowledgments}.
This work is supported by the International Research Fellowship of the Japanese Society for the Promotion of Science, together with the JSPS KAKENHI grant 15F15019. The author expresses his gratitude to A.\ Garroni, M.\ A.\ Peletier and L.\ Scardia for their substantial contributions to this work.



\newcommand{\etalchar}[1]{$^{#1}$}


\begin{thebibliography}{ADLGP16}

\bibitem[AA20]{AroraAcharya20}
R.~Arora and A.~Acharya.
\newblock A unification of finite deformation ${J}_2$ {V}on-{M}ises plasticity
  and quantitative dislocation mechanics.
\newblock {\em ArXiv: 2004.05647}, 2020.

\bibitem[Ach01]{Acharya01}
A.~Acharya.
\newblock A model of crystal plasticity based on the theory of continuously
  distributed dislocations.
\newblock {\em Journal of the Mechanics and Physics of Solids}, 49(4):761--784,
  2001.

\bibitem[ACH{\etalchar{+}}05]{AlvarezCarliniHochLBouarMonneau05}
O.~Alvarez, E.~Carlini, P.~Hoch, Y.~Le~Bouar, and R.~Monneau.
\newblock Dislocation dynamics described by non-local hamilton--jacobi
  equations.
\newblock {\em Materials Science and Engineering: A}, 400:162--165, 2005.

\bibitem[ADLGP14]{AlicandroDeLucaGarroniPonsiglione14}
R.~Alicandro, L.~De~Luca, A.~Garroni, and M.~Ponsiglione.
\newblock Metastability and dynamics of discrete topological singularities in
  two dimensions: a {$\Gamma$}-convergence approach.
\newblock {\em Archive for Rational Mechanics and Analysis}, 214(1):269--330,
  2014.

\bibitem[ADLGP16]{AlicandroDeLucaGarroniPonsiglione16}
R.~Alicandro, L.~De~Luca, A.~Garroni, and M.~Ponsiglione.
\newblock Dynamics of discrete screw dislocations on glide directions.
\newblock {\em Journal of the Mechanics and Physics of Solids}, 92:87--104,
  2016.

\bibitem[AO05]{ArizaOrtiz05}
M.~P. Ariza and M.~Ortiz.
\newblock Discrete crystal elasticity and discrete dislocations in crystals.
\newblock {\em Archive for Rational Mechanics and Analysis}, 178(2):149--226,
  2005.

\bibitem[BBP17]{BerendsenBurgerPietschmann17}
J.~Berendsen, M.~Burger, and J.~Pietschmann.
\newblock On a cross-diffusion model for multiple species with nonlocal
  interaction and size exclusion.
\newblock {\em Nonlinear Analysis}, 159:10--39, 2017.

\bibitem[BM17]{BlassMorandotti17}
T.~Blass and M.~Morandotti.
\newblock Renormalized energy and {P}each-{K}{\"o}hler forces for screw
  dislocations with antiplane shear.
\newblock {\em Journal of Convex Analysis}, 24(2):547--570, 2017.

\bibitem[CAWB06]{CaiArsenlisWeinbergerBulatov06}
W.~Cai, A.~Arsenlis, C.R. Weinberger, and V.V. Bulatov.
\newblock A non-singular continuum theory of dislocations.
\newblock {\em Journal of the Mechanics and Physics of Solids}, 54(3):561--587,
  2006.

\bibitem[CGO15]{ContiGarroniOrtiz15}
S.~Conti, A.~Garroni, and M.~Ortiz.
\newblock The line-tension approximation as the dilute limit of linear-elastic
  dislocations.
\newblock {\em Archive for Rational Mechanics and Analysis}, 218(2):699--755,
  2015.

\bibitem[CL05]{CermelliLeoni06}
P.~Cermelli and G.~Leoni.
\newblock Renormalized energy and forces on dislocations.
\newblock {\em SIAM Journal on Mathematical Analysis}, 37(4):1131--1160, 2005.

\bibitem[CP18]{CanizoPatacchini18DOI}
J.~A. Ca{\~n}izo and F.~S. Patacchini.
\newblock Discrete minimisers are close to continuum minimisers for the
  interaction energy.
\newblock {\em Calculus of Variations and Partial Differential Equations},
  57(1):1--24, 2018.

\bibitem[CXZ16]{ChapmanXiangZhu15}
S.~J. Chapman, Y.~Xiang, and Y.~Zhu.
\newblock Homogenization of a row of dislocation dipoles from discrete
  dislocation dynamics.
\newblock {\em SIAM Journal on Applied Mathematics}, 76(2):750--775, 2016.

\bibitem[DFF13]{DiFrancescoFagioli13}
M.~Di~Francesco and S.~Fagioli.
\newblock Measure solutions for non-local interaction pdes with two species.
\newblock {\em Nonlinearity}, 26(10):2777--2808, 2013.

\bibitem[DFF16]{DiFrancescoFagioli16}
M.~Di~Francesco and S.~Fagioli.
\newblock A nonlocal swarm model for predators--prey interactions.
\newblock {\em Mathematical Models and Methods in Applied Sciences},
  26(02):319--355, 2016.

\bibitem[DLGP12]{DeLucaGarroniPonsiglione12}
L.~De~Luca, A.~Garroni, and M.~Ponsiglione.
\newblock {$\Gamma$}-convergence analysis of systems of edge dislocations: the
  self energy regime.
\newblock {\em Archive for Rational Mechanics and Analysis}, 206(3):885--910,
  2012.

\bibitem[EFK17]{EversFetecauKolokolnikov17}
J.~H.~M. Evers, R.~C. Fetecau, and T.~Kolokolnikov.
\newblock Equilibria for an aggregation model with two species.
\newblock {\em SIAM Journal on Applied Dynamical Systems}, 16(4):2287--2338,
  2017.

\bibitem[EHIM09]{ElHajjIbrahimMonneau09}
A.~El~Hajj, H.~Ibrahim, and R.~Monneau.
\newblock Dislocation dynamics: from microscopic models to macroscopic crystal
  plasticity.
\newblock {\em Continuum Mechanics and Thermodynamics}, 21(2):109--123, 2009.

\bibitem[FG07]{FocardiGarroni07}
M.~Focardi and A.~Garroni.
\newblock A {1D} macroscopic phase field model for dislocations and a second
  order {$\Gamma$}-limit.
\newblock {\em Multiscale Modeling \& Simulation}, 6(4):1098--1124, 2007.

\bibitem[FIM08]{ForcadelImbertMonneau08}
N.~Forcadel, C.~Imbert, and R.~Monneau.
\newblock {\em On the Notions of Solutions to Nonlinear Elliptic Problems:
  Results and Developments}, chapter Viscosity solutions for particle systems
  and homogenization of dislocation dynamics.
\newblock Department of Mathematics of the Seconda Universita di Napoli, 2008.

\bibitem[FIM09]{ForcadelImbertMonneau09}
N.~Forcadel, C.~Imbert, and R.~Monneau.
\newblock Homogenization of the dislocation dynamics and of some particle
  systems with two-body interactions.
\newblock {\em Discrete and Continuous Dynamical Systems A}, 23(3):785--826,
  2009.

\bibitem[FIM12]{ForcadelImbertMonneau12}
N.~Forcadel, C.~Imbert, and R.~Monneau.
\newblock Homogenization of accelerated {F}renkel-{K}ontorova models with $n$
  types of particles.
\newblock {\em Transactions of the American Mathematical Society},
  364(12):6187--6227, 2012.

\bibitem[Gin19]{Ginster19}
J.~Ginster.
\newblock Plasticity as the $\gamma$-limit of a two-dimensional dislocation
  energy: The critical regime without the assumption of well-separateness.
\newblock {\em Archive for Rational Mechanics and Analysis}, 233(3):1253--1288,
  2019.

\bibitem[GLP10]{GarroniLeoniPonsiglione10}
A.~Garroni, G.~Leoni, and M.~Ponsiglione.
\newblock Gradient theory for plasticity via homogenization of discrete
  dislocations.
\newblock {\em Journal European Mathematical Society}, 12(5):1231--1266, 2010.

\bibitem[GM06]{GarroniMueller06}
A.~Garroni and S.~M{\"u}ller.
\newblock A variational model for dislocations in the line tension limit.
\newblock {\em Archive for Rational Mechanics and Analysis}, 181(3):535--578,
  2006.

\bibitem[GPPS13]{GeersPeerlingsPeletierScardia13}
M.~G.~D. Geers, R.~H.~J. Peerlings, M.~A. Peletier, and L.~Scardia.
\newblock Asymptotic behaviour of a pile-up of infinite walls of edge
  dislocations.
\newblock {\em Archive for Rational Mechanics and Analysis}, 209:495--539,
  2013.

\bibitem[GvMPS16]{GarroniVanMeursPeletierScardia16}
A.~Garroni, P.~{v}an Meurs, M.~A. Peletier, and L.~Scardia.
\newblock Boundary-layer analysis of a pile-up of walls of edge dislocations at
  a lock.
\newblock {\em Mathematical Models and Methods in Applied Sciences},
  26(14):2735--2768, 2016.

\bibitem[GvMPS19]{GarroniVanMeursPeletierScardia19DOI}
A.~Garroni, P.~{v}an Meurs, M.~A. Peletier, and L.~Scardia.
\newblock Convergence and non-convergence of many-particle evolutions with
  multiple signs.
\newblock {\em Archive for Rational Mechanics and Analysis}, pages 1--46, 2019.
\newblock Published online.

\bibitem[Hal11]{Hall11}
C.~L. Hall.
\newblock Asymptotic analysis of a pile-up of regular edge dislocation walls.
\newblock {\em Materials Science and Engineering: A}, 530:144--148, 2011.

\bibitem[HCO10]{HallChapmanOckendon10}
C.~L. Hall, S.~J. Chapman, and J.~R. Ockendon.
\newblock Asymptotic analysis of a system of algebraic equations arising in
  dislocation theory.
\newblock {\em SIAM Journal on Applied Mathematics}, 70(7):2729--2749, 2010.

\bibitem[HHvM18]{HallHudsonVanMeurs18}
C.~L. Hall, T.~Hudson, and P.~van Meurs.
\newblock Asymptotic analysis of boundary layers in a repulsive particle
  system.
\newblock {\em Acta Applicandae Mathematicae}, 153(1):1--54, 2018.

\bibitem[HIM09]{HajjIbrahimMonneau09}
A.~El Hajj, H.~Ibrahim, and R.~Monneau.
\newblock Homogenization of dislocation dynamics.
\newblock In {\em IOP Conferences Series: Materials Science and Engineering},
  2009.

\bibitem[HL82]{HirthLothe82}
J.~P. Hirth and J.~Lothe.
\newblock {\em Theory of Dislocations}.
\newblock John Wiley \& Sons, New York, 1982.

\bibitem[HM17]{HudsonMorandotti17}
T.~Hudson and M.~Morandotti.
\newblock Qualitative properties of dislocation dynamics: collisions and
  boundary behaviour.
\newblock {\em SIAM Journal on Applied Mathematics}, 77(5):1678--1705, 2017.

\bibitem[KCO02]{KoslowskiCuitinoOrtiz02}
M.~Koslowski, A.~M. Cuitino, and M.~Ortiz.
\newblock A phase-field theory of dislocation dynamics, strain hardening and
  hysteresis in ductile single crystals.
\newblock {\em Journal of the Mechanics and Physics of Solids},
  50(12):2597--2635, 2002.

\bibitem[KvMar]{KimuraVanMeurs19acc}
M.~Kimura and P.~{v}an Meurs.
\newblock Regularity of the minimiser of one-dimensional interaction energies.
\newblock {\em ESAIM: Control, Optimisation and Calculus of Variations}, 2019,
  to appear.

\bibitem[LL16]{LauteriLuckhaus16ArXiv}
G.~Lauteri and S.~Luckhaus.
\newblock An energy estimate for dislocation configurations and the emergence
  of cosserat-type structures in metal plasticity.
\newblock {\em ArXiv:1608.06155}, 2016.

\bibitem[MP12]{MonneauPatrizi12}
R.~Monneau and S.~Patrizi.
\newblock Homogenization of the {P}eierls--{N}abarro model for dislocation
  dynamics.
\newblock {\em Journal of Differential Equations}, 253(7):2064--2105, 2012.

\bibitem[MPS17]{MoraPeletierScardia17}
M.~G. Mora, M.~A. Peletier, and L.~Scardia.
\newblock Convergence of interaction-driven evolutions of dislocations with
  {W}asserstein dissipation and slip-plane confinement.
\newblock {\em SIAM Journal on Mathematical Analysis}, 49(5):4149--4205, 2017.

\bibitem[MRS19]{MoraRondiScardia19}
M.~G. Mora, L.~Rondi, and L.~Scardia.
\newblock The equilibrium measure for a nonlocal dislocation energy.
\newblock {\em Communications on Pure and Applied Mathematics}, 72(1):136--158,
  2019.

\bibitem[MSZ14]{MullerScardiaZeppieri14}
S.~M{\"u}ller, L.~Scardia, and C.~I. Zeppieri.
\newblock Geometric rigidity for incompatible fields, and an application to
  strain-gradient plasticity.
\newblock {\em Indiana University Mathematics Journal}, pages 1365--1396, 2014.

\bibitem[Nab47]{Nabarro47}
F.~R.~N. Nabarro.
\newblock Dislocations in a simple cubic lattice.
\newblock {\em Proceedings of the Physical Society}, 59(2):256, 1947.

\bibitem[Pei40]{Peierls40}
R.~Peierls.
\newblock The size of a dislocation.
\newblock {\em Proceedings of the Physical Society}, 52(1):34--37, 1940.

\bibitem[vM15]{VanMeurs15}
P.~{v}an Meurs.
\newblock {\em Discrete-to-Continuum Limits of Interacting Dislocations}.
\newblock PhD thesis, Eindhoven University of Technology, 2015.

\bibitem[vM18]{vanMeurs18}
P.~van Meurs.
\newblock Many-particle limits and non-convergence of dislocation wall
  pile-ups.
\newblock {\em Nonlinearity}, 31:165--225, 2018.

\bibitem[vMM14]{VanMeursMuntean14}
P.~{v}an Meurs and A.~Muntean.
\newblock Upscaling of the dynamics of dislocation walls.
\newblock {\em Advances in Mathematical Sciences and Applications},
  24(2):401--414, 2014.

\bibitem[vMM19]{VanMeursMorandotti19}
P.~van Meurs and M.~Morandotti.
\newblock Discrete-to-continuum limits of particles with an annihilation rule.
\newblock {\em SIAM Journal on Applied Mathematics}, 79(5):1940--1966, 2019.

\bibitem[vMMP14]{VanMeursMunteanPeletier14}
P.~{v}an Meurs, A.~Muntean, and M.~A. Peletier.
\newblock Upscaling of dislocation walls in finite domains.
\newblock {\em European Journal of Applied Mathematics}, 25(6):749--781, 2014.

\bibitem[Zin16]{Zinsl16}
J.~Zinsl.
\newblock Geodesically convex energies and confinement of solutions for a
  multi-component system of nonlocal interaction equations.
\newblock {\em Nonlinear Differential Equations and Applications NoDEA},
  23(4):1--43, 2016.

\end{thebibliography}
\end{document}